\documentclass[11pt, a4paper,leqno]{amsart}
\usepackage{amsmath,amsthm,amscd,amssymb,amsfonts, amsbsy}
\usepackage{latexsym}
\usepackage{txfonts}
\usepackage{exscale}
\usepackage{mathrsfs}
\usepackage{graphicx}
\usepackage{huncial}
\usepackage[T1]{fontenc}
\usepackage[titletoc]{appendix}
\usepackage{bbm}
\usepackage[all]{xy}

\def \R{ \mathbb{R} }
\def \N{ \mathbb{N} }

\allowdisplaybreaks

\usepackage[colorlinks,citecolor=red,pagebackref,hypertexnames=false, breaklinks]{hyperref}
\usepackage[usenames,dvipsnames]{color}

\def\supp{\mathop\mathrm{\,supp\,}}
\def\dist{\mathop\mathrm{\,dist\,}}
 
\def\mol{\mathop\mathrm{\,mol\,}} 

\calclayout
\allowdisplaybreaks

    \def\XXint#1#2#3{{\setbox0=\hbox{$#1{#2#3}{\int}$}
    \vcenter{\hbox{$#2#3$}}\kern-.5\wd0}}

\theoremstyle{plain}
\newtheorem{thm}{Theorem}
\newtheorem{lem}[thm]{Lemma}
\newtheorem{cor}[thm]{Corollary}
\newtheorem{prop}[thm]{Proposition}

\theoremstyle{definition}
\newtheorem{defn}[thm]{Definition}

\newtheorem{ex}{Example}

\theoremstyle{remark}
\newtheorem{rem}[thm]{Remark}

 \newcommand{\abs}[1]{{\left\lvert{#1}\right\rvert}}
\newcommand{\norm}[1]{{\left\lVert{#1}\right\rVert}}
\newcommand{\br}[1]{{\left({#1}\right)}}

\numberwithin{equation}{section}
\numberwithin{thm}{section}

\author{Li Chen}
\address{Li Chen, Mathematical Sciences Institute, The Australian National University, Canberra ACT 0200, Australia}
\curraddr{Instituto de Ciencias Matem\'aticas CSIC-UAM-UC3M-UCM,
Consejo Superior de Investigaciones Cient{\'\i}ficas,
C/ Nicol\'as Cabrera, 13-15,
E-28049 Madrid, Spain}
\email{li.chen@icmat.es}

\begin{document}
\title{Hardy spaces on metric measure spaces with generalized sub-Gaussian heat kernel estimates}

\date{\today}

\begin{abstract} 
Hardy space theory has been studied on manifolds or metric measure spaces equipped with either Gaussian or sub-Gaussian heat kernel behaviour. However, there are natural examples where one finds a mix of both behaviour (locally Gaussian and at infinity sub-Gaussian) in which case the previous theory doesn't apply. Still we define molecular and square function Hardy spaces using appropriate scaling, and we show that they agree with Lebesgue spaces in some range. Besides, counterexamples are given in this setting that the $H^p$ space corresponding to Gaussian estimates may not coincide with $L^p$. As a motivation for this theory, we show that the Riesz transform maps our Hardy space $H^1$ into $L^1$.
\end{abstract}

\maketitle

\tableofcontents
%%%%%%%%%%%%%%%%%%%%%%%%%%%%%%%%%%%%%%%%%%%%%%%%%%%%%

\section{Introduction}

The study of Hardy spaces originated in the 1910's and at the very beginning was confined to Fourier series and complex 
analysis in one variable. Since 1960's, it has been transferred to real analysis in several variables, or more generally to 
analysis on metric measure spaces. There are many different equivalent definitions of Hardy spaces, which involve suitable maximal functions, the atomic decomposition, the molecular decomposition, singular integrals, square functions etc.
See, for instance, the classical references \cite{FS72, CW77,CMS85,St93}. 

More recently, a lot of work has been devoted to the theory of Hardy spaces associated with operators, see for example, 
\cite{AMR08,HLMMY11,U11,AMM13} and the references therein.

In \cite{AMR08}, Auscher, McIntosh and Russ studied Hardy spaces with respect to the Hodge Laplacian on Riemannian 
manifolds with the doubling volume property by using the Davies-Gaffney type estimates. They defined Hardy spaces of differential forms of all degrees via molecules 
and square functions, 
on which the Riesz transform is $H^p$ bounded for $1\le p\le \infty$. Comparing with the Lebesgue spaces, it holds that $H^p\subset L^p$ for $1\leq p\leq 2$ and $L^p \subset H^p$ for $p>2$. Moreover, under the assumption of Gaussian heat kernel upper bound, $H^p$ coincides $L^p$ for $1<p<\infty$.

In \cite{HLMMY11}, Hofmann, Lu, Mitrea, Mitrea and Yan further developed the theory of $H^1$ and $BMO$ spaces adapted to a metric measure space $(M,d,\mu)$ with the volume doubling property endowed with a non-negative self-adjoint operator $L$, which generates an analytic semigroup $\{e^{-tL}\}_{t>0}$ 
satisfying the so-called Davies-Gaffney estimate: there exist $C,c>0$ such that for any open sets $U_1,U_2\subset M$,
and for every $f_i \in L^2(M)$ with $\supp f_i \subset U_i$, $i=1,2$,
\begin{align}\label{DG-normal}
|<e^{-tL}f_1,f_2>| \leq C \exp\left(-\frac{\dist^2(U_1,U_2)}{ct}\right) \Vert f_1\Vert_{2} \Vert f_2\Vert_2,~\forall t>0,
\end{align}
 where $\dist(U_1,U_2):= \inf_{x\in U_1,y\in U_2}d(x,y)$. 
The authors extended results of \cite{AMR08} by obtaining an atomic decomposition of the $H^1$ space.

More generally, instead of \eqref{DG-normal}, if $M$ satisfies the Davies-Gaffney estimate of order $m$ with $m\geq 2$: for all $x,y \in M$ and for all $t>0$,
\begin{align}\label{DGm}
\norm{\mathbbm 1_{B(x,t^{1/m})} e^{-tL} \mathbbm 1_{B(y,t^{1/m})}}_{2\to2}
\leq C\exp\br{-c\br{\frac{d(x,y)}{t}}^{\frac{m}{m-1}}}.
\end{align}
Kunstmann and Uhl \cite{U11,KU15} defined Hardy spaces via square functions and via molecules adapted to \eqref{DGm}, where
the two $H^1$ spaces are also equivalent. 
%However, unlike \cite{HLMMY11}, it is not possible to get an equivalent $H^1$ space via atoms for $m>2$. 
 Here and in the sequel, $B(x,r)$ denotes the ball of centre $x\in M$ and radius $r>0$ and $V(x,r)=\mu(B(x,r))$. In addition, if the $L^{p_0}-L^{p'_0}$ off-diagonal estimates of order $m$ holds: for all $x,y \in M$ and for all $t>0$,
\begin{align}\label{DGp}
\norm{\mathbbm 1_{B(x,t^{1/m})} e^{-tL} \mathbbm 1_{B(y,t^{1/m})}}_{p_0 \to p'_0} 
\leq \frac{C}{V^{\frac{1}{p_0}-\frac{1}{p'_0}}(x,t^{1/m})} \exp\br{-c\br{\frac{d(x,y)}{t}}^{\frac{m}{m-1}}}
\end{align}
with $p_0'$ the conjugate of $p_0$, then the Hardy space $H^p$ defined via 
square functions coincides with $L^p$ for $p\in (p_0,2)$.

However, there are natural examples where, one finds a mix of both behaviours \eqref{DG-normal} and \eqref{DGm}, in which case the previous Hardy space theory doesn't apply. For example, on fractal manifolds, the heat kernel behaviour is locally Gaussian and at infinity  sub-Gaussian (see Section \ref{HK estimates} for more details). We aim to develop a proper Hardy space theory for this setting. An important motivation for our Hardy spaces theory is to study the Riesz transform on fractal manifolds, where the weak type $(1,1)$ boundedness has recently been proved in a joint work by the author with Coulhon, Feneuil and Russ \cite{CCFR15}.

In this paper, we work on doubling metric measure spaces endowed with a non-negative self-adjoint operator which satisfy the doubling volume property and  the $L^2$ off-diagonal estimate with different local and global decay (see \eqref{DG} below). The specific description will be found below in Section \ref{setting}. We define two classes of Hardy spaces in this setting, via molecules and via conical square functions, see Setion \ref{definitions}. Both definitions have the scaling adapted to the off-diagonal decay \eqref{DG}. 

In Section 3, we identify the two different $H^1$ spaces. The molecular $H^1$ spaces are always convenient spaces to deal with Riesz transform and other sub-linear operators, while the $H^p$, $p \ge1$, spaces defined via conical square functions possess certain good properties like real and complex interpolation. The identification of both spaces gives us a powerful tool to study the Riesz transform, Littlewood-Paley functions, boundary value problems for elliptic operators etc. 

In Section 4, we compare the Hardy spaces defined via conical square functions with the Lebesgue spaces. Assuming further an $L^{p_0}-L^{p_0'}$ off-diagonal estimate for some $1\le p_0<2$ with different local and global decay for the heat semigroup, we show the equivalence of our $H^p$ spaces and the Lebesgue spaces $L^p$ for $p_0< p<p_0'$. We also justify that the scaling for the Hardy spaces is the right one, by disproving this equivalence of $H^p$ and $L^p$ for $p$ close to $2$ on some fractal Riemannian manifolds. As far as we know, no previous results are known in this direction.

In Section 5, we shall apply our theory to prove that the Riesz transform is $H^1-L^1$ bounded on fractal manifolds. The proof is inspired by \cite{CCFR15} (see \cite{Fe15}for the original proof in the discrete setting), where the integrated estimate for the gradient of the heat kernel plays a crucial role.

In the following, we will introduce our setting, the definitions and the main results more specifically.

\medskip

\noindent{\bf Notation} Throughout this paper, we denote $u\simeq v$ if $v\lesssim u$ and $u\lesssim v$, where $u\lesssim v$ means that there exists a constant $C$ (independent of the important parameters) such that $u\leq Cv$.

For a ball $B\subset M$ with radius $r>0$ and given $\alpha>0$, we
write $\alpha B$ as the ball with the same centre and the radius $\alpha r$.
We denote $C_1(B)=4B$, and $C_j(B)=2^{j+1} B\backslash 2^{j}B$ for $j\geq 2$.

\subsection{The setting}\label{setting}
We shall assume that $M$ is a metric measure space satisfying the doubling volume property: for any $x\in M$ and $r>0,
$\begin{align}
\label{doubling}V(x,2r)\lesssim V(x,r)\tag{$D$}
\end{align}
and the $L^2$ Davies-Gaffney estimate with different local and global decay for the analytic semigroup $\{e^{-tL}\}_{t>0}$ generated by the non-negative self-adjoint operator $L$, that is,
$\forall x,y\in M$,
\begin{align}\label{DG}\tag{$DG_{\rho}$}
\norm{\mathbbm 1_{B(x,t)} e^{-\rho(t)L} \mathbbm 1_{B(y,t)}}_{2\to 2} 
\lesssim \left\{
\begin{aligned}
& \exp\br{-c\br{\frac{d(x,y)}{t}}^{\frac{\beta_1}{\beta_1-1}}} & 0<t<1, \\
& \exp\br{-c\br{\frac{d(x,y)}{t}}^{\frac{\beta_2}{\beta_2-1}}}, & t\geq 1,
\end{aligned}\right.
\end{align}
where $1<\beta_1\leq\beta_2$ and
\begin{align}\label{rho}
\rho(t)=\left\{ \begin{aligned}
         &t^{\beta_1},&0<t<1, \\
         & t^{\beta_2},&t\geq 1.
         \end{aligned}\right.
         \end{align}

Recall a simple consequence of \eqref{doubling}: there exists $\nu>0$ such that 
 \begin{align} \label{D1}
\frac{V(x,r)}{V(x,s)}\lesssim \br{\frac{r}{s}}^\nu,\,\,\forall x\in M, \,r\geq s>0.\end{align}
It follows that
\[
V(x,r)\lesssim \br{1+\frac{d(x,y)}{r}}^{\nu} V(y,r),\,\,\forall x\in M, \,r\geq s>0.
\]
Therefore,
\begin{align}\label{D2}
\int_{d(x,y)<r}\frac{1}{V(x,r)}d\mu (x)\simeq  1,\,\,\forall y\in M, \,r>0.
\end{align}
If $M$ is non-compact, we also have a reverse inequality of \eqref{D1} (see for instance \cite[p. 412]{Gr09}). That is, there exists $\nu'>0$ such that
\begin{align}\label{revdb}
\frac{V(x,r)}{V(x,s)}\gtrsim \left(\frac{r}{s}\right)^{\nu'},
\,\,\forall x\in M, \,r\geq s>0.
\end{align}
 
Also notice that in \eqref{rho}, if necessary we may smoothen $\rho(t)$ as 
\[
\rho(t)=\left\{ \begin{aligned}
       &t^{\beta_1},&\text{if } 0<t\leq 1/2, \\
	&\text{smooth part},&\text{if } 1/2<t<2, \\
       &t^{\beta_2},&\text{if } t\geq 2;
       \end{aligned}\right.
\]
with $\rho'(t)\simeq 1$ for $1/2<t<2$, which we still denote by $\rho(t)$. Since $\frac{\rho'(t)}{\rho(t)} =\frac{\beta_1}{t}$ for 
$0<t\leq 1/2$ and $\frac{\rho'(t)}{\rho(t)} =\frac{\beta_2}{t}$ for $t\geq 2$, we have 
in a uniform way
\begin{align}\label{der}
\frac{\rho'(t)}{\rho(t)} \simeq \frac{1}{t}.
\end{align}

We say that $M$ satisfies an $L^{p_0}-L^{p'_0}$ off-diagonal estimate for some $1<p_0<2$ if
\begin{align}\label{DG'}\tag{$DG_{\rho}^{p_0}$}
\norm{\mathbbm 1_{B(x,t)} e^{-\rho(t)L} \mathbbm 1_{B(y,t)}}_{p_0\to p_0'} 
\lesssim \left\{
\begin{aligned}
&\frac{1}{V^{\frac{1}{p_0}-\frac{1}{p_0'}}(x,t)} \exp\br{-c\br{\frac{d(x,y)}{t}}^{\frac{\beta_1}{\beta_1-1}}} & 0<t<1, \\
&\frac{1}{V^{\frac{1}{p_0}-\frac{1}{p_0'}}(x,t)} \exp\br{-c\br{\frac{d(x,y)}{t}}^{\frac{\beta_2}{\beta_2-1}}}, & t\geq 1,
\end{aligned}\right.
\end{align}
and a generalized pointwise sub-Gaussian heat kernel estimate if for all $x,y\in M$,
\begin{align}\label{ue}\tag{$U\!E_{\rho}$}
p_{\rho(t)}(x,y)
\lesssim \left\{
\begin{aligned}
&\frac{1}{V(x,t)} \exp\br{-c\br{\frac{d(x,y)}{t}}^{\frac{\beta_1}{\beta_1-1}}} & 0<t<1, \\
&\frac{1}{V(x,t)} \exp\br{-c\br{\frac{d(x,y)}{t}}^{\frac{\beta_2}{\beta_2-1}}}, & t\geq 1,
\end{aligned}\right.
\end{align}
Examples of fractal manifolds satisfy \eqref{ue} with $\beta_1=2$ and $\beta_2>2$, see Section 2 below for more information.

%%%%%%%%%%%%%%%%%%%%%%%%%%%%%%%%%%%%%%%%%%%%%%%%%%%%%%%%%%%%%%%%%%

\subsection{Definitions}\label{definitions}        
Recall that 
\begin{defn} \label{mol}
Let $\varepsilon >0$ and an integer $K$ be an integer such that $K>\frac{\nu}{2\beta_1}$, where $\nu$ is in \eqref{D1}. A function $a\in L^2(M)$ is called a $(1,2,\varepsilon )-$molecule associated to $L$ 
if there exist a function $b \in \mathcal{D}(L)$ and a ball $B$ with radius $r_B$ such that
\begin{enumerate}
\item $a = L^K b$;
\item It holds that for every $k=0,1, \cdots, K$ and $i=0,1,2,\cdots$, we have
 \begin{align}\label{molb}
\Vert(\rho (r_{B})L )^k b\Vert_{L^2(C_i(B))} \leq \rho^K(r_{B})2^{-i\varepsilon } V(2^i B)^{-1/2}.
\end{align}
\end{enumerate}
\end{defn}

\begin{defn}\label{molh}
We say that $f=\sum_{n=0}^{\infty }\lambda _n a_n$ is a molecular $(1,2,\varepsilon )-$representation of $f$ if 
$(\lambda_n)_{n\in \mathbb{N}}\in l^1$, each $a_n$ is a molecule as above, and the sum converges in the $L^2$ sense. 
We denote the collection of all the functions with a molecular representation by $\mathbb{H}_{L,\rho,\mol}^1$, 
where the norm of $f\in \mathbb{H}_{L,\rho,\mol}^1$ is given by
\[
\Vert f\Vert_{\mathbb{H}_{L,\rho,\mol}^1(M)}=\inf \left \{ \sum_{n=0}^{\infty }|\lambda _n|:
f=\sum_{n=0}^{\infty }\lambda _n a_n \text{ is a molecular } (1,2,\varepsilon )-\text{representation} \right\}.
\]
The Hardy space $H_{L,\rho,\mol}^1(M)$ is defined as the completion of $\mathbb{H}_{L,\rho,mol}^1(M)$ with respect to this 
norm.
\end{defn}

Consider the following conical square function
\begin{align}\label{SFrho}
S_h^{\rho} f(x) =\br{\iint_{\Gamma (x)}|\rho(t)L e^{-\rho(t)L }f(y)|^2\frac{d\mu(y)}{V(x,t)}\frac{dt}{t}}^{1/2},
\end{align}
where the cone $\Gamma(x)=\{(y,t)\in M\times (0,\infty ): d(y,x)<t\}$.

We define first the $L^2(M)$ adapted Hardy space $H^2(M)$ as the closure of the range of $L $ in $L^2(M)$ norm, i.e., $H^2(M):=
\overline{R(L )}$. 

\begin{defn}
The Hardy space $H_{L,S_h^{\rho}}^p(M)$, $p\geq 1$ is defined as the completion of the set
$\{f\in H^2(M): \Vert S_h^{\rho} f\Vert_{L^p}<\infty \}$ with respect to the norm $\Vert S_h^{\rho} f\Vert_{L^p}$.
The $H_{L,S_h^{\rho}}^p(M)$ norm is defined by
$\Vert f\Vert_{H_{L,S_h^{\rho}}^p(M)}:=\Vert S_h^{\rho} f\Vert_{L^p(M)}$.
\end{defn}
For $p=2$, the operator $S_h^{\rho}$ is bounded on $L^2(M)$. Indeed, for every $f\in L^2(M)$,
\begin{align}\label{L2}
\begin{split}
\Vert S_h^{\rho} f\Vert_{L^2(M)}^2
&= \int_M\iint_{\Gamma (x)} \abs{\rho(t)L e^{-\rho(t)L }f(y)}^2 \frac{d\mu(y)}{V(x,t)}\frac{dt}{t}d\mu (x)
\\ &\simeq \iint_{M\times (0,\infty)} \abs{\rho(t)L e^{-\rho(t)L }f(y)}^2 d\mu(y)\frac{dt}{t}
\\ & \simeq \iint_{M\times (0,\infty)} \abs{\rho(t)L e^{-\rho(t)L }f(y)}^2 d\mu(y)\frac{\rho'(t)dt}{\rho(t)}
\\ & = \int_0^{\infty} <(\rho(t) L)^{2} e^{-2\rho(t)L }f,f> \frac{\rho'(t)dt}{\rho(t)}
\simeq \Vert f\Vert_{L^2(M)}^2.
\end{split}
\end{align}
Note that the second step follows from Fubini theorem and \eqref{D2} in Section 2.3. The third step is obtained by using the fact \eqref{der}: $ \rho'(t)/\rho(t) \simeq 1/t$. The last one is a consequence of spectral theory. 

\begin{rem}\label{rem:SF}
The above definitions are similar as in \cite{HLMMY11} (also \cite{AMR08} for $1$-forms on Riemannian manifolds) and \cite{KU15,U11}. The difference is that we replace $t^2$ or $t^m$ by $\rho(t)$ in \eqref{molb} and \eqref{SFrho}. 

In the case when $\rho(t)=t^2$, we denote $S_h^\rho$ by $S_h$, that is,
\begin{align}\label{SF-normal}
S_h f(x) :=\left(\iint_{\Gamma (x)}|t^2 L e^{-t^2 L }f(y)|^2\frac{d\mu(y)}{V(x,t)}\frac{dt}{t}\right)^{1/2}, 
\end{align}
and denote $H_{L,S_h^\rho}^p$ by  $H_{L,S_h}^p$.

\end{rem}
%%%%%%%%%%%%%%%%%%%%%%%%%%%%%%%%%%%%%%%%%%%%%%%%%%%%%%%%%%%%%%%%%

\subsection{Main results}

We first obtain the equivalence between $H^1$ spaces defined via molecules and via square functions.
 \begin{thm}\label{H1equiv}
Let $M$ be a metric measure space satisfying the doubling volume property \eqref{doubling} and the $L^2$
off-diagonal heat kernel estimate \eqref{DG}.
Then $H_{L,\rho,\mol}^1(M)=H_{L,S_h^{\rho}}^1(M)$, which we denote by $H_{L,\rho}^1(M)$. Moreover,
$$\Vert f\Vert_{H_{L,\rho,\mol}^1(M)} \simeq \Vert f\Vert_{H_{L,S_h^{\rho}}^1(M)}.$$
\end{thm}

Now compare $H_{L,S_h^{\rho}}^p(M)$ and $L^p$ for $1< p < \infty$.

Recall that on Riemannian manifold satisfying the doubling volume property \eqref{doubling} and the Gaussian upper bound for the heat kernel of the operator, we have 
$H_{L,S_h}^p(M)=L^p(M)$, $1< p <\infty$, see for example \cite[Theorem 8.5]{AMR08} for Hardy spaces of $0-$forms on 
Riemannian manifold. However, in general, the equivalence is not known.
It is also proved in \cite{KU15,U11} that if the $L^{p_0}-L^{p'_0}$ off-diagonal estimates of order $m$ \eqref{DGp} holds, 
then the Hardy space $H_{S_h^m}^p$ (see Remark \ref{rem:SF}) coincides with $L^p$ for $p\in (p_0,2)$.

Our result in this direction is the following:
\begin{thm}\label{equihl}
Let $M$ be a non-compact metric measure space as above. Let $1\leq p_0<2$ and $\rho$ be as above. Suppose that $M$ 
satisfies \eqref{doubling} and \eqref{DG'}. 
Then $H_{L,S_h^{\rho}}^p(M)=L^p(M)$ for $p_0<p<p_0'$.
\end{thm}

If one assumes the pointwise heat kernel estimate, then Theorems 1.5 and 1.6 yield the following.
\begin{cor}
Let $M$ be a non-compact metric measure space satisfying the doubling volume property \eqref{doubling} and the pointwise heat kernel estimate \eqref{ue}.
Then $H_{L,\rho,\mol}^1(M)=H_{L,S_h^{\rho}}^1(M)$, and $H_{L,S_h^{\rho}}^p(M)=L^p(M)$ for $1<p<\infty$.
\end{cor}

In the following theorem, we show that for $1<p<2$, the equivalence may not hold between $L^p$ and $H^p$ defined via 
conical square function $S_h$ with scaling $t^2$. The counterexamples we find are certain Riemannian manifolds satisfying \eqref{doubling} and two-sided sub-Gaussian heat kernel estimate: \eqref{ue} and its reverse, with $\beta_1=2$ and $\beta_2=m>2$. Notice that in this case, $L$ is the non-negative Laplace-Beltrami operator, which we denote by $\Delta$. For simplicity, we denote \eqref{ue} by $(U\!E_{2,m})$ and the two sided estimate by $(HK_{2,m})$. Also, we denote by $H_{\Delta,m,mol}^1$ the $H^1$ space defined via molecules $H_{L,\rho,mol}^1$, $H_{\Delta,S_h^m}^{p}$ the $H^p$ space defined via square functions $H_{L,S_h^{\rho}}^{p}$. 

\begin{thm}\label{noequiv}
Let $M$ be a Riemannian manifold with polynomial volume growth
\begin{align}\label{d}
V(x,r) \simeq r^d,\,\, r\geq 1,
\end{align}
as well as two-sided sub-Gaussian heat kernel estimate ($H\!K_{2,m}$) with $2<m<d/2$, that is, $(U\!E_{2,m})$ and 
the matching lower estimate. Then 
$$L^p(M)\subset H_{\Delta,S_h}^p(M)$$ 
doesn't hold for $p\in \br{\frac{d}{d-m},2}$.
\end{thm}

As an application of this Hardy space theory, we have
\begin{thm}
\label{thm2}
Let $M$ be a manifold satisfying the doubling volume property \eqref{doubling} and the heat kernel estimate $(U\!E_{2,m})$, $m>2$, 
that is, the upper bound of $(H\!K_{2,m})$. 
Then the Riesz transform $\nabla \Delta^{-1/2}$ is $H_{\Delta,m}^1-L^1$ bounded.
\end{thm}

\begin{rem}
Recall that under the same assumptions, it is proved in \cite{CCFR15} that the Riesz transform is of weak type $(1,1)$ and thus $L^p$ bounded for $1<p<2$. 
\end{rem}

%%%%%%%%%%%%%%%%%%%%%%%%%%%%%%%%%%%%%%%%%%%%%%%%%%%%%%%%%%%%%%%

\section{Preliminaries}

%
%%%%%%%%%%%%%%%%%%%%%%%%%%%%%%%%%%%%%%%%%%%%%%%%%%%%%%%%%%%%%%%%%

\subsection{More about sub-Gaussian off-diagonal and pointwise heat kernel estimates}\label{HK estimates}

Let us first give some examples that satisfy \eqref{DG'} with $\beta_1\neq \beta_2$.  More examples of this case are metric measure Dirichlet spaces, which we refer to \cite{Ba13,Stu95,Stu94,HSC01} for details. 

\begin{ex}\label{fm}
Fractal manifolds.

Fractal manifolds are built from graphs with a self-similar structure at infinity by replacing the edges of the graph with tubes of length $1$ and then gluing the tubes together smoothly at the vertices. For instance, see \cite{BCG01} for the construction of Vicsek graphs. For any $D,m\in \R$ such that $D> 1$ and  $2< m\leq D+1$, there exist complete connected Riemannian manifolds satisfying $V(x,r) \simeq r^D$ for $r\geq 1$ and \eqref{ue} with $\beta_1=2$ and $\beta_2=m>2$ in \eqref{rho} (see \cite{Ba04} and \cite{CCFR15}).
\end{ex}

\begin{ex}\label{cs}
Cable systems (Quantum graphs) (see \cite{V85}, \cite[Section 2]{BB04}). 

Given a weighted graph $(G,E,\nu)$, we define the cable system $G_C$ by replacing each edge of $G$ by a copy of $(0,1)$ 
joined together at the vertices. The measure $\mu$ on $G_C$ is given by $d\mu(t)=\nu_{xy} dt$ for $t$ in the cable connecting 
$x$ and $y$, and $\mu$ assigns no mass to any vertex. The distance between two points $x$ and $y$ is given as follows: 
if $x$ and $y$ are on the same cable, the length is just the usual Euclidean distance $|x-y|$. If they are on different cables, 
then the distance is $\min\{|x-z_x|+d(z_x,z_y)+|z_y-y|\}$ ($d$ is the usual graph distance), where the minimum is taken over all 
vertices $z_x$ and $z_y$ such that $x$ is on a cable with one end at $z_x$ and $y$ is on a cable with one end at $z_y$. 
One takes as the core $\mathcal C$ the functions in $C(G_C)$ which have compact support and are $C^1$ on each cable, 
and sets
\[
\mathcal E(f,f):=\int_{G_C} \abs{f'(t)}^2 d\mu(t).
\]
Let $L$ be the associated non-negative self-adjoint operator 
associated with $\mathcal E$ and $\{e^{-tL}\}_{t>0}$ be the generated semigroup. Then the associated kernel may satisfies \eqref{ue}. 
For example, the cable graph associated with the Sierpinski gasket graph (in $\mathbb Z^2$) satisfies $(U\!E_{2,\log 5/\log2})$.
\end{ex}

%%%%%%%%%%%%%%%%%%%%%%%%%%%%%%%%%%%%%%%%%%%%%%%%%%%%%%%%%%%%%%%%%

The following are some useful lemmas for the off-diagonal estimates. We first observe that \eqref{ue} $\Rightarrow$ \eqref{DG'} $\Rightarrow$ \eqref{DG} for $1\leq p_0\leq 2$. Indeed,
\begin{lem}[\cite{BK05}]\label{BK'}
Let $(M,d,\mu)$ be a metric measure space satisfying the doubling volume property. Let $L$ be a non-negative self-adjoint operator on $L^2(M,\mu)$. Assume that 
\eqref{DG'} holds. Then for all $p_0\leq u \leq v \leq p_0'$, we have
\begin{align}
\norm{\mathbbm 1_{B(x,t)} e^{-\rho(t)L} \mathbbm 1_{B(y,t)}}_{u\to v} 
\lesssim \left\{
\begin{aligned}
&\frac{1}{V^{\frac{1}{u}-\frac{1}{v}}(x,t)} \exp\br{-c\br{\frac{d(x,y)}{t}}^{\frac{\beta_1}{\beta_1-1}}} & 0<t<1, \\
&\frac{1}{V^{\frac{1}{u}-\frac{1}{v}}(x,t)} \exp\br{-c\br{\frac{d(x,y)}{t}}^{\frac{\beta_2}{\beta_2-1}}}, & t\geq 1.
\end{aligned}\right.
\end{align}
\end{lem}

\begin{rem}
The estimate \eqref{DG'} is equivalent to the $L^{p_0}-L^{2}$ off-diagonal estimate
\[
\norm{\mathbbm 1_{B(x,t)} e^{-\rho(t)L} \mathbbm 1_{B(y,t)}}_{p_0\to 2} 
\lesssim \left\{
\begin{aligned}
&\frac{1}{V^{\frac{1}{p_0}-\frac{1}{2}}(x,t)} \exp\br{-c\br{\frac{d(x,y)}{t}}^{\frac{\beta_1}{\beta_1-1}}} & 0<t<1, \\
&\frac{1}{V^{\frac{1}{p_0}-\frac{1}{2}}(x,t)} \exp\br{-c\br{\frac{d(x,y)}{t}}^{\frac{\beta_2}{\beta_2-1}}}, & t\geq 1.
\end{aligned}\right.
\]
We refer to \cite{BK05,CS08} for the proof.
\end{rem}

In fact, we also have
\begin{lem}[\cite{BK05,U11}]\label{BK}
Let $(M,d,\mu)$ satisfy \eqref{doubling}. Let $L$ be a non-negative self-adjoint operator on $L^2(M,\mu)$. Assume that 
\eqref{DG'} holds. Then for all $p_0\leq u \leq v \leq p_0'$ and $k \in \mathbb N$, we have
\begin{enumerate}
\item For any ball $B\subset M$ with radius $r>0$, and any $i\geq 2$,
\begin{align}\label{DG2'}
\norm{\mathbbm 1_{B}(tL)^k e^{-tL} \mathbbm 1_{C_i(B)}}_{u\to v}, 
\norm{\mathbbm 1_{C_i(B)}(tL)^k e^{-tL} \mathbbm 1_{B}}_{u\to v} \lesssim
 \left\{\begin{aligned}
&\frac{2^{i\nu}}{\mu^{\frac1{u}-\frac1{v}}(B)} e^{-c\br{\frac{2^{i\beta_1} r^{\beta_1}}{t}}^{1/(\beta_1-1)}} & 0<t<1, \\
&\frac{2^{i\nu}}{\mu^{\frac1{u}-\frac1{v}}(B)} e^{-c\br{\frac{2^{i\beta_2} r^{\beta_2}}{t}}^{1/(\beta_2-1)}}, & t\geq 1.
\end{aligned}\right.
\end{align}

\item For all $\alpha,\beta\geq 0$ such that $\alpha+\beta=\frac1{u}-\frac1{v}$,
\[
\norm{V^{\alpha}(\cdot,t) (\rho(t)L)^k e^{-\rho(t)L} V^{\beta}(\cdot, t)}_{u\to v} \leq C.
\]
\end{enumerate}
\end{lem}

%%%%%%%%%%%%%%%%%%%%%%%%%%%%%%%%%%%%%%%%%%%%%%%%%%%%%%%%%

\subsection{Tent spaces}
We recall definitions and properties related to tent spaces on metric measure spaces with the doubling volume property, following \cite{CMS85}, \cite{Ru07}.

Let $M$ be a metric measure space satisfying \eqref{doubling}. For any $x\in M$ and for any closed subset $F \subset M$, a saw-tooth region is defined as $\mathcal R(F) :=\bigcup _{x\in F}
\Gamma(x)$. If $O$ is an open subset of $M$, then the ``tent" over $O$, denoted by $\widehat{O}$, is defined as
\[
\widehat{O}:=[\mathcal R(O^c)]^c=\{(x,t)\in M \times (0,\infty ): d(x,O^c)\geq t\}.
\]

For a measurable function $F$ on $M\times (0,\infty )$, consider
\[
\mathcal A F(x)
=\left(\iint_{\Gamma (x)}|F(y,t)|^2\frac{d\mu(y)}{V(x,t)}\frac{dt}{t}\right)^{1/2}.
\]
Given $0 < p < \infty $, say that a measurable function $F\in T_2^p(M\times (0,\infty ))$ if
\[
\Vert F\Vert_{T_2^p(M)}:=\Vert \mathcal A F\Vert_{L^p(M)}<\infty.
\]
For simplicity, we denote $T_2^p(M\times (0,\infty) )$ by $T_2^p(M)$ from now on. 

Therefore, for $f\in H_{L,S_h}^p(M)$ and $0<p<\infty$, write $ F(y,t)=\rho(t) L e^{-\rho(t)L }f(y)$, we have
\[
\Vert f\Vert_{H_{L,S_h^{\rho}}^p(M)}=\Vert F\Vert_{T_2^p(M)}.
\]

Consider another functional
\[
\mathcal CF(x)
=\sup_{x\in B}\left(\iint_{\widehat B}|F(y,t)|^2\frac{d\mu(y)dt}{t}\right)^{1/2},
\]
we say that a measurable function $F\in T_2^{\infty}(M)$ if $\mathcal C F\in L^{\infty}(M)$.

\begin{prop}
Suppose $1<p<\infty $, let $p'$ be the conjugate of $p$. Then the pairing 
$<F,G>\longrightarrow \int_{M\times (0,\infty )}F(x, t)G(x,t)\frac{d\mu (x)dt}{t}$ realizes $T_2^{p'}(M)$ as the dual of $T_2^p(M)$.
\end{prop}

Denote by $[\,,]_\theta $ the complex method of interpolation described in \cite{BL76}. Then we have the following result of 
interpolation of tent spaces, where the proof can be found in \cite{Am14}.
\begin{prop}
\label{inter}
Suppose $1\leq p_0 < p < p_1\leq \infty $, with $1/p = (1 - \theta )/p_0 +
\theta /p_1$ and $0 < \theta < 1$. Then
\[
[T_2^{p_0}(M),T_2^{p_1}(M)]_\theta =T_2^{p}(M).
\]
\end{prop}

Next we review the atomic theory for tent spaces which was originally developed in \cite{CMS85}, and extended to the setting of spaces of 
homogeneous type in \cite{Ru07}.
\begin{defn}
A measurable function $A$ on $M \times (0,\infty )$ is said to be a $T_2^1-$atom if there exists a ball $B \in M$ such that $A$ is 
supported in $\widehat{B}$ and
\[
\int_{M\times (0,\infty )}|A(x, t)|^2 d\mu(x)\frac{dt}{t} \leq \mu^{-1}(B).
\]
\end{defn}

\begin{prop}[\cite{HLMMY11},\cite{Ru07}]
\label{decomp1}
For every element $F\in T_2^1 (M)$ there exist a sequence of numbers $\{\lambda _j\}_{j=0}^{\infty }\in l^1$ and a sequence of 
$T_2^1-$atoms $\{A_j\}_{j=0}^{\infty }$ such that
\begin{align}\label{convg}
F =\sum_{j=0}^{\infty }\lambda _j A_j \text{ in } T_2^1 (M) \text{ and a.e. in } M \times (0,\infty ).
\end{align}
Moreover, $\sum_{j=0}^{\infty }\lambda _j \approx \Vert F\Vert_{T_2^1(M)}$, where the implicit constants depend only on the 
homogeneous space properties of $M$.

Finally, if $F \in T_2^1(M) \cap  T_2^2(M)$, then the decomposition \eqref{convg} also converges in $T_2^2(M)$.
\end{prop}

%%%%%%%%%%%%%%%%%%%%%%%%%%%%%%%%%%%%%%%%%%%%%%%%%%%%%%%%%%

%%%%%%%%%%%%%%%%%%%%%%%%%%%%%%%%%%%%%%%%%%%%%%%%%%%%%
\medskip
\section{The molecular decomposition}
In this section, we shall prove Theorem \ref{H1equiv}. That is, under the assumptions of \eqref{doubling} and \eqref{DG},
the two $H^1$ spaces: $H_{L,\rho,\mol}^1(M)$ and $H_{L, S_h^{\rho}}^{1}(M)$, are equivalent. We denote 
\[
H_{L,\rho}^{1}(M):=H_{L, S_h^{\rho}}^{1}(M)= H_{L,\rho,\mol}^1(M).
\]

Since $H_{L,\rho,\mol}^1(M)$ and $H_{L,S_h^{\rho}}^1(M)$ are 
completions of $\mathbb H_{L,\rho,\mol}^1(M)$ and $H_{L,S_h^{\rho}}^1(M)\cap H^2(M)$, it is enough to show 
$\mathbb H_{L,\rho,\mol}^1(M)= H_{L,S_h^{\rho}}^1(M)\cap H^2(M)$ with equivalent norms. In the following, we will prove the 
two-sided inclusions seperately. Before proceeding to the proof, we first note the lemma below to prove 
$H_{L,\rho,\mol}^1(M)-L^1(M)$ boundedness of an operator, which is an analogue of Lemma 4.3 in \cite{HLMMY11}.

\begin{lem}\label{crit}
Assume that $T$ is a linear operator, or a nonnegative sublinear operator, satisfying the weak-type $(2,2)$ bound 
\begin{align}
\mu\left(\left\{x\in M:  |Tf(x)| >\eta \right\}\right) \lesssim \eta^{-2} \Vert f\Vert_2^2, ~~\forall \eta>0
\end{align}
and that for every $(1,2,\varepsilon)-$molecule $a$, we have
\begin{align}\label{ta}
\Vert Ta\Vert_{L^1}\leq C,
\end{align}
with constant $C$ independent of $a$. Then $T$ is bounded from $\mathbb H_{L,\rho,\mol}^1(M)$ to $L^1(M)$ with
\[
\Vert Tf\Vert_{L^1}\lesssim \Vert f\Vert_{\mathbb H_{L,\,\rho,\mol}^1(M)}.
\]
Consequently, by density, $T$ extends to be a bounded operator from $H_{L,\rho,\mol}^1(M)$ to $L^1(M)$.
\end{lem}

For the proof, we refer to \cite{HLMMY11}, which is also applicable here. 

\medskip

%%%%%%%%%%%%%%%%%%%%%%%%%%%%%%%%%%%%%%%%%%%%%%%%%%%%%%%%%%
\subsection{The inclusion $\mathbb H_{L,\,\rho,\mol}^1(M) \subseteq H_{L,\,S_h^{\rho}}^1(M)\cap H^2(M)$.} 
We have the following theorem:
\begin{thm}
Let $M$ be a metric measure space satisfying the doubling volume property \eqref{doubling} and the heat kernel estimate \eqref{DG}. 
Then $\mathbb H_{L,\,\rho,\mol}^1(M) \subseteq H_{L,\,S_h^{\rho}}^1(M)\cap H^2(M)$ and
\[
\Vert f\Vert_{H_{L,\,S_h^{\rho}}^1(M)} \leq C \Vert f\Vert_{\mathbb H_{L,\,\rho,\mol}^1(M)}.
\]
\end{thm}

\begin{proof}
First observe that $\mathbb H_{L,\,\rho,\mol}^1(M) \subseteq H^2(M)$. Indeed, by Definition \ref{mol}, any 
$(1,2,\varepsilon)$-molecule belongs to $R(L)$. Thus any finite linear combination of molecules belongs to 
$R(L)$. Since $f \in \mathbb H_{L,\,\rho,\mol}^1(M)$ is the $L^2(M)$ limit of finite linear combination of molecules, 
we get $f\in\overline{R(L)} = H^2(M)$.

It remains to show $\mathbb H_{L,\,\rho,\mol}^1(M) \subseteq H_{L,\,S_h^{\rho}}^1(M)$, that is, $S_h^{\rho}$ is 
bounded from $\mathbb H_{L,\,\rho,\mol}^1(M)$ to $L^1(M)$. Note that $S_h^{\rho}$ is $L^2$ bounded by spectral theory (see \eqref{L2}), it follows from Lemma \ref{crit} that it suffices to prove that, for any $(1,2,\varepsilon )$-molecule $a$, 
there exists a constant $C$ such that $\Vert S_h^{\rho} a\Vert_{L^1(M)}\leq C$. In other words, one needs to prove 
$\Vert A \Vert_{T_2^1(M)}\leq C$, where
\[
A(y,t)=\rho(t) L e^{-\rho(t)L }a(y).
\]

Assume that $a$ is a $(1,2,\varepsilon)$-molecule related to a function $b$ and a ball $B$ with radius $r$, that is, $a=L^K b$ and for every $k=0,1,\cdots, K$ and 
$i=0,1,2,\cdots$, it holds that
\[
\Vert(\rho(r)L )^k b\Vert_{L^2(C_i(B))} \leq \rho(r)2^{-i\varepsilon} \mu(2^i B)^{-1/2}.
\]

Similarly as in \cite{AMR08}, we divide $A$ into four parts:
\begin{align*}
A &=  \mathbbm 1_{2B\times (0,2r)}A +\sum_{i\geq 1}\mathbbm 1_{C_{i}(B)\times (0,r)}A
+\sum_{i\geq 1}\mathbbm 1_{C_{i}(B)\times (r,2^{i+1}r)}A
+\sum_{i\geq 1}\mathbbm 1_{2^i B\times (2^i r,2^{i+1}r)}A
\\ &=: A_0+A_1+A_2+A_3.
\end{align*}
Here $\mathbbm 1 $ denotes the characteristic function and $C_i(B)=2^{i+1}B \backslash 2^i(B)$, $i\geq 1$. 
It suffices to show that for every $j=0,1,2,3$, we have $\Vert A_j\Vert_{T_2^1}\leq C$.

Firstly consider $A_0$. Observe that 
\[
\mathcal A(A_0)(x)
=\br{\iint_{\Gamma (x)}\abs{\mathbbm 1_{2 B\times (0,2r)}(y,t)A(y,t)}^2 \frac{d\mu(y)}{V(x,t)}\frac{dt}{t}}^{1/2}
\] 
is supported on $4B$. Indeed, denote by $x_B$ be the center of $B$, then 
$d(x,x_B)\leq d(x,y)+d(y,x_B)\leq 4r$.
Also, it holds that
\begin{align*}
\norm{A_0}_{T_2^2(M)}^2 &= \norm{\mathcal A(A_0)}_{2}^2
\leq 
\int_{M} \iint_{\Gamma (x)} \abs{\rho(t) L e^{-\rho(t)L }a(y)}^2 \frac{d\mu (y)}{V(x,t)}\frac{dt}{t}d\mu(x)
\\ &\lesssim  
 \Vert a \Vert_{L^2(M)}^2
 \lesssim \mu^{-1}(B).
\end{align*}
Here the second and the third inequalities follow from \eqref{L2} and the definition of molecules respectively.
 Now applying the Cauchy-Schwarz inequality, then
\[
\Vert A_0 \Vert_{T_2^1(M)}\leq \Vert A \Vert_{T_2^2(M)} \mu(4B)^{1/2}\leq C.
\]

Secondly for $A_1$. For each $i\geq 1$, we have $\supp \mathcal A(\mathbbm 1 _{C_{i}(B)\times (0,r)}A)\subset 2^{i+2}B$. 
In fact, $d(x,x_B)\leq d(x,y)+d(y,x_B) \leq t+2^{i+1}r < 2^{i+2}r$. Then
\begin{align*}
\norm{\mathbbm 1_{C_{i}(B)\times (0,r)}A}_{T_2^2}
&= \norm{\mathcal A(\mathbbm 1_{C_{i}(B)\times (0,r)}A)}_{2}
\\ &\leq  
\br{\int_{2^{i+2}B} \iint_{\Gamma(x)} \abs{\mathbbm 1_{C_{i}(B)\times (0,r_B)}(y,t) \rho(t)L e^{-\rho(t)L} a(y)}^2
\frac{d\mu(y)}{V(x,t)}\frac{dt}{t}d\mu(x)}^{1/2}
\\ &\leq  
\br{\int_0^{r} \int_{C_{i}(B)} \abs{\rho(t)L e^{-\rho(t)L }a(y)}^2 d\mu(y)\frac{dt}{t}}^{1/2}
\\ &\leq  
\sum_{l=0}^{\infty } \br{\int_0^{r} \int_{C_{i}(B)} \abs{\rho(t)L e^{-\rho(t)L } \mathbbm 1_{C_l(B)}a(y)}^2 d\mu(y)\frac{dt}{t}}^{1/2}
\\ &=: 
\sum_{l=0}^{\infty } I_l.
\end{align*}

We estimate $I_l$ with $|i-l|>3$ and $|i-l|\leq 3$ respectively. Firstly assume that $|i-l|\leq 3$. Using \eqref{L2} again, we have
\[
I_l^2 \leq 
\int_0^\infty \int_{M}\abs{\rho (t)L e^{-\rho(t)L} \mathbbm 1_{C_l(B)} a(y)}^2 d\mu (y)\frac{dt}{t}
\lesssim \Vert a \Vert_{L^2(C_l(B))}^2
\lesssim   2^{-2 i\varepsilon} \mu^{-1}(2^i B).
\]

Assume now $|i-l|>3$. Note that $\dist(C_l(B),C_{i}(B))\geq c 2^{\max\{l,i\}}r_B \geq c 2^{i}r_B$. Then it follows from 
Lemma \ref{BK} that
\begin{align}\label{Il}
\begin{split}
I_l^2 &\leq
\int_0^{r} \exp\br{-c\br{\frac{\rho(2^i r)}{\rho(t)}}^{\frac{\beta_2}{\beta_2-1}}} \norm{a}_{L^2(C_l(B))}^2 d\mu(y)\frac{dt}{t}
\\ &\lesssim 
2^{-2l\varepsilon}\mu^{-1}(2^l B) \int_0^{r}\br{\frac{\rho(t)}{\rho(2^i r)}}^{c} \frac{dt}{t}
\lesssim
 2^{-ci} 2^{-2l\varepsilon}\mu^{-1}(2^i B).
\end{split}
\end{align}
%Here $\tau(t)=\frac{\beta_1}{\beta_1-1}$ if $0<t<1$, otherwise $\tau(t)=\frac{\beta_2}{\beta_2-1}$. In the second inequality, we dominate $\exp\br{-c\br{\frac{2^i r}{t}}^{\tau(t)}}$ by $\br{\frac{t}{2^i r}}^{c+\nu}$. 
The last inequality follows from \eqref{D1}.
 
It follows from above that
\begin{align*}
\norm{\mathbbm 1_{C_{i}(B)\times (0,r)}A}_{T_2^2}
\lesssim 
\sum_{l:|l-i|\leq 3}2^{-i\varepsilon}\mu^{-1/2}(2^i B)+\sum_{l:|l-i|>3}2^{-ic} 2^{-l\varepsilon}\mu^{-1/2}(2^i B)
\lesssim  2^{-ic}\mu^{-1/2}(2^i B),
\end{align*}
where $c$ depends on $\varepsilon, M$. Therefore
\[
\norm{A_1}_{T_2^1}
\leq  \sum_{i\geq 1} \norm{\mathbbm 1_{C_{i}(B)\times (0,r)}A}_{T_2^2}\mu^{1/2}(2^{i+2}B)
\lesssim \sum_{i\geq 1}2^{-ic}\leq C.
\]

We estimate $A_2$ in a similar way as before except that we replace $a$ by $L^K b$.  Note that for each $i\leq 1$, we have 
$\supp \mathcal A(\mathbbm 1_{C_{i}(B)\times (r,2^{i+1}r)}A)\subset 2^{i+2}B$. Indeed, 
\[
d(x,x_B)\leq d(x,y)+d(y,x_B) \leq t+2^{i+1}r \leq 2^{i+2}r.
\] 
Then
\begin{align*}
\norm{\mathbbm 1_{C_{i}(B)\times (r,2^{i+1}r)}A}_{T_2^2} 
&= \norm{\mathcal A(\mathbbm 1_{C_{i}(B)\times (r,2^{i+1}r)}A)}_{2} 
\\ &\leq  
\br{\int_{2^{i+2} B} \iint_{\Gamma (x)}\abs{\mathbbm 1_{C_i(B)\times (r,2^i r)}(y,t)A(y,t)}^2 
\frac{d\mu(y)dt}{V(x,t)t}d\mu(x)}^{1/2}
\\ &\leq  
\br{\int_{r}^{2^{i+1}r} \int_{C_i(B)} \abs{(\rho(t)L)^{K+1} e^{-\rho(t)L}b(y)}^2 d\mu(y)\frac{dt}{t\rho^{2K}(t)}}^{1/2}
\\ &\leq  
\br{\sum_{l=0}^{\infty}\int_{r}^{2^{i+1}r} \int_{C_i(B)}\abs{(\rho(t)L)^{K+1} e^{-\rho(t)L}\mathbbm 1_{C_l(B)}b(y)}^2 
d\mu(y)\frac{dt}{t\rho^{2K}(t)}}^{1/2}
\\ &=:
\sum_{l=0}^{\infty } J_l
\end{align*}
When $|i-l|\leq 3$, by spectral theorem we get $J_l^2 \leq C 2^{-2i\varepsilon }V^{-1}(2^i B)$.
And when $|i-l|>3$, it holds $\dist(C_l(B),C_{i}(B))\geq c 2^{\max\{l,i\}}r \geq c 2^{i}r$. Then we estimate $J_l$ in the same way as for \eqref{Il},
\begin{align*}
J_l^2 
&\leq
\int_{r}^{2^{i+1}r} \exp\br{-c\br{\frac{\rho(2^i r)}{\rho(t)}}^{\frac{\beta_2}{\beta_2-1}}} \Vert b\Vert_{L^2(C_l(B))}^2 d\mu(y) \frac{d t}{t\rho^{2K}(t)}
\\ &\leq
\rho^{2K}(r)2^{-2 l\varepsilon}\mu^{-1}(2^{l+1}B)\int_{r}^{2^{i+1}r} \br{\frac{\rho(t)}{\rho(2^i r)}}^{c} \frac{dt}{t\rho^{2K}(t)}
\\ &\lesssim 
 2^{-ic}2^{-l(2\varepsilon+\nu)}\mu^{-1}(2^{i}B).
\end{align*}
Here $c$ in the second and the third lines are different. We can carefully choose $c$ in the second line to make sure that $c$ in the third line is positive. 
%we take $\nu<c<2\beta_1K$ (properly chosen).
%The same estimate also holds for $l>i+3$, in the same way as \eqref{Il'}.

Hence
\[
\norm{\mathbbm 1_{C_i(B)\times (r,2^i r)}A}_{T_2^2}^2 \lesssim 2^{-ic}\mu^{-1}(2^i B),
\]
and
\[
\norm{A_2}_{T_2^1}
\leq  \sum_{i\geq 1} \norm{\mathbbm 1_{C_i(B)\times(r,2^i r)}A}_{T_2^2}\mu^{1/2}(2^{i+2}B)
\lesssim \sum_{i\geq 1}2^{-ic/2}\leq C.
\]

It remains to estimate the last term $A_3$. For each $i\geq 1$, we still have 
\[
\supp \mathcal A(\mathbbm 1_{2^i B\times (2^i r,2^{i+1}r)}A)\subset 2^{i+2}B.
\] 
Then we obtain as before that
\begin{align*}
\norm{\mathbbm 1_{2^i B\times (2^i r,2^{i+1}r)}A}_{T_2^2}
&= \norm{\mathcal A(\mathbbm 1_{2^i B\times (2^i r,2^{i+1}r)} A)}_{2}
\\ &\leq  
\br{\int_{2^{i+2}B} \iint_{\Gamma (x)}\abs{\mathbbm 1_{2^i B\times (2^i r_B,2^{i+1}r)}(y,t)A(y,t)}^2
\frac{d\mu(y) dt}{V(x,t)t}d\mu (x)}^{1/2}
\\ &\leq  
\br{\int_{2^i r}^{2^{i+1}r} \int_{2^i B}\abs{(\rho(t) L)^{K+1} e^{-\rho(t)L }b(y)}^2\frac{d\mu(y)dt}{t\rho^{2K}(t)}}^{1/2}
\\ &\leq  
\sum_{l=0}^{\infty} \br{\int_{2^i r}^{2^{i+1}r} \int_{2^i B}\abs{(\rho(t) L)^2
e^{-\rho(t)L} \mathbbm 1_{C_l(B)}b(y)}^2\frac{d\mu(y)dt}{t\rho^{2K}(t)}}^{1/2}
\\& =:  \sum_{l=0}^{\infty} K_l.
\end{align*}
In fact, due to the doubling volume property, \eqref{L2} as well as the definition of molecules, we get
\begin{align*}
K_l^2 &\leq  
\int_{2^i r}^{2^{i+1}r} \norm{\mathbbm 1_{C_l(B)}b}_{L^2}^2 \frac{dt}{t\rho^{2K}(t)}
\lesssim
\rho^{2K}(r) 2^{-2l\varepsilon}\mu^{-1}(2^l B)\int_{2^i r}^{2^{i+1}r}\frac{dt}{t\rho^{2K}(t)}
\\ &\lesssim  2^{-2l\varepsilon}2^{-ic}\mu^{-1}(2^{i}B).
\end{align*}

Hence
\[
\Vert A_3\Vert_{T_2^1}
\leq  \sum_{i\geq 1} \Vert \mathbbm 1 _{2^i B\times (2^i r,2^{i+1}r)} A\Vert_{T_2^2}\mu^{1/2}(2^{i+2}B)
\lesssim \sum_{i\geq 1}2^{-2i}\leq C.
\]

This finishes the proof.
\end{proof}

%%%%%%%%%%%%%%%%%%%%%%%%%%%%%%%%%%%%%%%%%%%%%%%%%%%%%

\subsection{The inclusion $H_{L,\,S_h^{\rho}}^1(M) \cap H^2(M) \subseteq \mathbb H_{L,\,\rho,\mol}^1(M)$.}
We closely follow the proof of Theorem 4.13 in \cite{HLMMY11} and get
\begin{thm}\label{cnt}
Let $M$ be a metric measure space satisfying \eqref{doubling} and \eqref{DG}. If $f \in H_{L ,S_h^{\rho}}^1(M)\cap  H^2(M)$, then there exist a sequence of numbers 
$\{\lambda _j\}_{j=0}^{\infty }\subset l^1$ and a sequence of $(1,2,\varepsilon )-$molecules $\{a_j\}_{j=0}^{\infty }$ such that $f$ 
can be represented in the form $f =\sum_{j=0}^{\infty }\lambda _j a_j$, with the sum converging in $L^2(M)$, and
\[
\Vert f\Vert_{\mathbb H_{L,\,\rho,\mol}^1(M)}\leq C\sum_{j=0}^{\infty }\lambda _j
\leq C \Vert f\Vert_{H_{L,\,S_h^{\rho}}^1(M)},
\]
where $C$ is independent of $f$. In particular,
$H_{L,\,S_h^{\rho}}^1(M)\cap H^2(M)\subseteq \mathbb H_{L,\,\rho,\mol}^1(M)$.
\end{thm}

\begin{proof} For $f \in H_{L,\,S_h^{\rho}}^1(M)\cap  H^2(M)$, denote
$F(x,t)=\rho(t)L e^{-\rho(t)L}f(x)$. Then by the definition of $H_{L ,S_h^{\rho}}^1(M)$, we have $F\in T_2^1(M) \cap 
T_2^2(M)$.

From Theorem \ref{decomp1}, we decompose $F$ as $F =\sum_{j=0}^{\infty }\lambda _j A_j$, where
$\{\lambda _j\}_{j=0}^{\infty }\in l^1$, $\{A_j\}_{j=0}^{\infty }$ is a sequence of $T_2^1-$atoms
supported in a sequence of sets $\{\widehat B_j\}_{j=0}^{\infty }$, and the sum converges in both $T_2^1(M)$ and $T_2^2(M)$. Also
\[
\sum_{j=0}^{\infty }\lambda _j \lesssim \Vert F\Vert_{T_2^1(X)}= \Vert f\Vert_{H_{L ,S_h^{\rho}}^1(M)}.
\]

For $f\in H^2(M)$, by functional calculus, we have the following ``Calder\'on reproducing formula"
\[
f = C\int_0^{\infty } (\rho(t)L)^{K+1} e^{-2\rho (t)L}f \frac{\rho'(t)dt}{\rho(t)}
  = C\int_0^{\infty } (\rho(t)L)^K e^{-\rho(t)L}F(\cdot ,t) \frac{\rho'(t)dt}{\rho(t)}
  =: C \pi _{h,L} (F).
\]

Denote $a_j=C\pi _{h,L} (A_j)$, then $f =\sum_{j=0}^{\infty }\lambda _j a_j$. Since for  $F\in T_2^2(M)$, we have $\Vert \pi _{h,
L} (F)\Vert_{L^2(M)}\leq C \Vert F\Vert_{T_2^2(M)}$. Thus we learn from Lemma 4.12 in \cite{HLMMY11} that the sum also 
converges in $L^2(M)$. 

We claim that $a_j, j=0,1,...,$ are $(1,2,\varepsilon )-$molecules up to multiplication to some uniform constant. 

Indeed, note that $a_j=L^K b_j$, where 
$$b_j=C\int_0^{\infty } \rho^K(t) e^{-\rho(t)L}A_j(\cdot ,t) \frac{\rho'(t)dt}{\rho(t)}.$$
Now we estimate the norm $\Vert (\rho(r_{B_j})L )^k b_j\Vert_{L^2(C_i(B))}$, where $r_{B_j}$ is the radius of $B_j$. For simplicity we ignore the index $j$.
Consider any function $g\in L^2(C_i(B))$ with $\Vert g\Vert_{L^2(C_i(B))}=1$, then for $k=0,1, \cdots, K$,
\begin{align*}
& \left|\int_M (\rho(r_B)L )^k b(x)g(x)d\mu (x)\right|
\\ & \lesssim  
 \abs{\int_{M}\br{\int_0^{\infty }(\rho(r_B)L )^k \rho^K (t) e^{-\rho (t)L}(A_j(\cdot ,t))(x)\frac{\rho'(t)dt}{\rho(t)}}g(x)d\mu (x)}
\\  &= 
 \abs{\int_{\widehat B}\br{\frac{\rho(r_B)}{\rho(t)}}^k \rho^K(t) A_j(x ,t)(\rho (t)L)^k e^{-\rho(t)L}g(x)d\mu(x) \frac{\rho'(t)dt}{\rho(t)}}
\\ &\lesssim 
 \br{\int_{\widehat B}\abs{A_j(x,t)}^2 d\mu (x)\frac{dt}{t}}^{1/2}
\br{\int_{\widehat B}\abs{\br{\frac{\rho(r_B)}{\rho (t)}}^k \rho^K(t) (\rho(t)L)^k e^{-\rho (t)L}g(x)}^2 d\mu (x) \frac{dt}{t}}^{1/2}.
\end{align*}
In the last inequality, we apply H\"older inequality as well as \eqref{der}.

We continue to estimate by using the definition of $T_2^1-$atoms and the off-diagonal estimates of heat kernel. 

For $i=0,1$, the above quantity is dominated by
\[
 \mu^{-1/2}(B) \rho(r_B) \left(\int_{\widehat B}\left|(\rho (t)L )^k e^{-\rho (t)L}g(x)\right|^2 d\mu (x)\frac{dt}{t}\right)^{1/2}
\lesssim \mu^{-1/2}(B) \rho(r_B) .
\]
Next for $i\geq 2$, the above estimate is controlled
\begin{align*}
&  \mu^{-1/2}(B) \br{\int_0^{r_B}\br{\frac{\rho(r_B)}{\rho (t)}}^{2 k}\rho^{2K}(t)
\norm{(\rho(t)L)^k e^{-\rho(t)L}g}_{L^2(B)}^2\frac{dt}{t}}^{1/2}
\\ & \lesssim 
 \mu^{-1/2}(B) \br{\int_0^{r_B}\br{\frac{\rho(r_B)}{\rho(t)}}^{2 k} \rho^{2K}(t) \exp\br{-c\br{\frac{2^i r_B}{t}}^\tau} \frac{dt}{t}}^{1/2}
\\ & \lesssim 
 \mu^{-1/2}(B) 
\br{\int_0^{r_B}\br{\frac{\rho (r_B)}{\rho(t)}}^{2k}\rho^{2K}(t) \br{\frac{t}{2^i r_B}}^{\varepsilon +\nu}\frac{dt}{t}}^{1/2}
\\ & \lesssim  
 \mu^{-1/2}(2^i B) \rho^K(r_B) 2^{-i \varepsilon }.
\end{align*}
In the first inequality, we use Lemma \ref{BK}. Since $k=0, 1, \cdots, K$, the last inequality always holds for any $\varepsilon >0$.

Therefore,
\begin{align*}
\Vert (\rho (r_{B})L )^k b\Vert_{L^2(C_i(B))}
& =  \sup_{\Vert g\Vert_{L^2(C_i(B))}=1} \abs{\int_M (\rho(r_B)L )^k b(x)g(x)d\mu (x)}
\\ &\lesssim \mu^{-1/2}(2^i B) \rho^K(r_B)2^{-i \varepsilon }.
\end{align*}
\end{proof}

%%%%%%%%%%%%%%%%%%%%%%%%%%%%%%%%%%%%%%%%%%%%%%%%%%%%%%%%%%%%%%%%
\section{Comparison of Hardy spaces and Lebesgue spaces}

In this section, we will study the relations between $L^p(M)$, $H_{L,\,S_h^{\rho}}^p(M)$ and $H_{L,\,S_h}^p(M)$ under the 
assumptions of \eqref{doubling} and \eqref{DG'}. We first show that $L^p(M)$ and $H_{L,\,S_h^{\rho}}^p(M)$ are 
equivalent. Next we give some examples such that $L^p(M)$ and $H_{L,\,S_h}^p(M)$ are not equivalent. More precisely, the inclusion $L^p\subset H_{L,\,S_h}^p$ may be false for $1<p<2$. 

%%%%%%%%%%%%%%%%%%%%%%%%%%%%%%%%%%%%%%%%%%%%%%%%%%%%%%%%%%%%%%%%%

\subsection{Equivalence of $L^p(M)$ and $H_{L,\,S_h^{\rho}}^p(M)$ for $p_0<p<p_0'$}

We will prove Theorem \ref{equihl}. That is, if $M$ satisfies \eqref{doubling} and \eqref{DG'}, then $H_{L,\,S_h^{\rho}}^p(M)=L^p(M)$ for $p_0<p<p_0'$.

Our main tool is the Calder\'on-Zygmund decomposition (see for example \cite[Corollaire 2.3]{CW71}).
\begin{thm}\label{C-Z}
Let $(M,d,\mu)$ be a metric measured space satisfying the doubling volume property. Let $1\leq q\leq \infty$ and $f\in L^q$. 
Let $\lambda>0$. Then there exists a decomposition of $f$, $f=g+b=g+\sum_i b_i $ so that

\begin{enumerate}
\item $|g(x)|\leq C\lambda$  for almost all $x\in M$;

\item There exists a sequence of balls $B_i =B(x_i,r_i)$ so that each $b_i$ is supported in $B_i$,
\[
\int| b_i(x)|^q d\mu(x)\leq C\lambda^q \mu(B_i)
\]

\item $\sum_i \mu(B_i)\leq \frac{C}{\lambda^q} \int|f(x)|^q d\mu(x)$;

\item $\Vert b\Vert_q \leq C\Vert f\Vert_q$ and $\Vert g\Vert_q \leq C\Vert f\Vert_q$;

\item There exists $k\in \mathbb{N}^*$ such that each $x\in M$ is contained in at most $k$ balls $B_i$.
\end{enumerate}
\end{thm}

\medskip
 
\begin{proof}[Proof of Theorem \ref{equihl}]
Due to the self-adjointness of $L$ in $L^2(M)$, we get
$L^2(M)=\overline{R(L )}\bigoplus N(L)$, where the sum is orthogonal. Under the assumptions \eqref{doubling} and \eqref{DG'}, 
we have $N(L)=0$ and thus $H^2(M)=L^2(M)$.  Indeed, for any $f\in N(L)$, it holds
\[
e^{-\rho(t)L}f-f=\int_0^{\rho(t)} \frac{\partial}{\partial s}e^{-sL}fds=-\int_0^{\rho(t)} L e^{-sL}fds=0,
\]
 As a consequence of Lemma \ref{BK'}, we have that for all $x\in M$ and $t\geq 0$,
\begin{eqnarray*}
\br{\int_{B(x,t)} |f|^{p_0'} }^{1/p_0'}
=\norm{e^{-\rho(t)L}f}_{L^{p_0'}(B(x,t))}  \lesssim V(x,t)^{\frac{1}{p_0'}-\frac{1}{2}}\Vert f\Vert_{L^2(B(x,t))}.
\end{eqnarray*}
Now letting $t\rightarrow \infty$, we obtain that $f=0$.
 
It suffices to prove that for any $f\in R(L)\cap L^p(M)$ with $p_0<p<p_0'$,
\begin{align}\label{equi}
\Vert S_h^{\rho} f\Vert_{L^p}\lesssim \Vert f\Vert_{L^p}.
\end{align}
With this fact at hand, we can obtain by duality that $\Vert f\Vert_{L^p} \leq C \Vert S_h^{\rho} f\Vert_{L^p}$ for $p_0<p< p_0'$. 

Indeed, for $f \in R(L)$, write the identity
\[
f = C\int_0^{\infty }(\rho(t)L )^2 e^{-2\rho(t)L}f\frac{\rho'(t)dt}{\rho(t)},
\]
where the integral $C\int_{\varepsilon}^{1/\varepsilon}(\rho(t)L )^2 e^{-2\rho (t)L}f\frac{\rho'(t)dt}{\rho(t)}$ converges to $f$ 
in $L^2(M)$ as $\varepsilon\rightarrow 0$.

Then for $f\in R(L)\cap L^p(M)$, we have
\begin{align*}
\Vert f\Vert_{L^p}
& =\sup_{\norm{g}_{L^{p'}}\leq 1}|<f,g>|
\simeq \sup_{\norm{g}_{L^{p'}}\leq 1}\abs{\iint_{M\times (0,\infty )} F(y,t)G(y,t)d\mu(y)\frac{\rho'(t)dt}{\rho(t)}}
\\& \simeq
\sup_{\norm{g}_{L^{p'}}\leq 1} \abs{\int_M \iint_{\Gamma (x)} F(y,t)G(y,t)\frac{d\mu(y)}{V(x,t)}\frac{\rho'(t)dt}{\rho(t)}d\mu(x)}
\\&\lesssim
 \sup_{\norm{g}_{L^{p'}}\leq 1} \norm{F}_{T_2^p} \norm{G}_{T_2^{p'}}
\simeq \sup_{\norm{g}_{L^{p'}}\leq 1} \norm{S_h f}_{L^p} \norm{S_h g}_{L^{p'}}
\\&\lesssim
 \sup_{\norm{g}_{L^{p'}}\leq 1} \norm{S_h f}_{L^p} \norm{g}_{L^{p'}}
=\norm{S_h f}_{L^p}.
\end{align*}
Here $F(y,t)=\rho(t)L e^{-\rho(t)L}f(y)$ and $G(y,t)=\rho(t)L e^{-\rho(t)L} g(y)$. The second line's equivalence is 
due to the doubling volume propsimeqerty.

By an approximation process, the above argument holds for $f\in L^p(M)$.

For $p>2$, the $L^p$ norm of the conical square function is controlled by its vertical analogue (for a reference, see \cite{AHM12}, 
where the proof can be adapted to the homogenous setting), which is always $L^p$ bounded for $p_0<p<p_0'$ by 
adapting the proofs in \cite{Bl07} and \cite{CDMY96} (if $\{e^{-tL}\}_{t>0}$ is a symmetric Markov semigroup, then it is $L^p$
bounded for $1<p<\infty$, according to \cite{St70}). Hence 
\eqref{equi} holds. 

It remains to show \eqref{equi} for $p_0<p< 2$.

In the following, we will prove the weak $(p_0,p_0)$ boundedness of $S_h^{\rho}$ by using the Calder\'on-Zygmund decomposition. 
Since $S_h^{\rho}$ is also $L^2$ bounded as shown in \eqref{L2}, then by interpolation, \eqref{equi} holds for every $p_0<p<2$. 
The proof is similar to \cite[Proposition 6.8]{Au07} and \cite[Theorem 3.1]{AHM12}, which originally comes from \cite{DM99}.

We take the Calder\'on-Zygmund decomposition of $f$ at height $\lambda $, that is, $f=g+\sum b_i$ with 
$\supp b_i\subset B_i$. Since $S_h^{\rho}$ is a sublinear operator, write
\begin{align*}
S_h^{\rho} \left(\sum_i b_i \right)
&= S_h^{\rho} \br{\sum_i  \br{I-\br{I-e^{-\rho (r_i)}}^N+\br{I-e^{-\rho (r_i)L }}^N}b_i}
\\ &\leq
S_h^{\rho} \br{\sum_i \br{I-\br{I-e^{-\rho (r_i)}}^N} b_i}+S_h^{\rho} \br{\sum_i \br{I-e^{-\rho (r_i)}}^N b_i }.
\end{align*}
Here $N \in \mathbb N$ is chosen to be larger than $2\nu /\beta_1$, where $\nu$ is as in \eqref{D1}.

Then it is enough to prove that
\[
\begin{split}
& \mu \br{\left\{x\in M: S_h^{\rho}(f)(x)>\lambda \right\}}
\leq  \mu \br{\left\{x\in M: S_h^{\rho}(g)(x)>\frac{\lambda}{3}\right\}}
\\ &+\mu \br{\left\{x\in M: S_h^{\rho} \br{\sum_i \br{I-\br{I-e^{-\rho (r_i)}}^N} b_i} (x)>\frac{\lambda}{3}\right\}}
\\& + \mu \br{\left\{x\in M: S_h^{\rho} \br{\sum_i \br{I-e^{-\rho (r_i)L }}^N b_i}(x)>\frac{\lambda}{3} \right\}}
\lesssim \frac{1}{\lambda^{p_0}}\int |f(x)|^{p_0}d\mu(x).
\end{split}
\]

We treat $g$ in a routine way. Since $S_h^{\rho}$ is $L^2$ bounded as shown in \eqref{L2}, then 
\[
\mu \br{\left\{x\in M: S_h^{\rho}(g)(x)>\frac{\lambda}{3}\right\}}
\lesssim \lambda^{-2}\norm{g}_2^2 \lesssim \lambda^{-p_0}\norm{g}_{p_0}
\lesssim \lambda^{-p_0}\Vert f\Vert_{p_0}.
\]

Now for the second term. Note that $I-\br{I-e^{-\rho (r_i)L}}^N = \sum_{k=1}^{N} (-1)^{k+1} \binom{N}{k} e^{-k\rho (r_i)L}$, it is enough to show that for every $1 \leq k \leq N$, 
\begin{align}\label{k}
\mu \br{\left\{x\in M: S_h^{\rho} \left( \sum_i e^{-k\rho (r_i)L} b_i \right)(x)>\frac{\lambda}{3N} \right\}}
\lesssim \frac{1}{\lambda^{p_0}} \int |f(x)|^{p_0} d\mu(x).
\end{align}

Note the following slight improvement of \eqref{DG2'}: for every $1 \leq k\leq N$ and for every $j\geq 1$, we have 
\begin{align}\label{off}
\norm{e^{-k \rho(r_i)L}b_i}_{L^2(C_j(B_i))} \lesssim \frac{2^{j\nu}}{\mu^{\frac1{p_0}-\frac12}(B_i)} e^{-c_k 2^{j\tau(k \rho(r_i))}}
\norm{b_i}_{L^{p_0}(B_i)}.
\end{align}
Here $\tau(r)=\beta_1/(\beta_1-1)$ if $0<r<1$, otherwise $\tau(r)=\beta_2/(\beta_2-1)$. Indeed, it is obvious for $r_i\ge 1$ and $0<r_i <k^{-\frac1{\beta_1}}$. For $k^{-\frac1{\beta_1}} \le r_i<1$, that is, $k\rho(r_i)\ge 1$, then $\br{\frac{(2^j r_i)^{\beta_2}}{k\rho(r_i)}}^{\frac{1}{\beta_2-1}} \simeq 2^{j\frac{\beta_2}{\beta_2-1}}=2^{j\tau(k \rho(r_i))}$.

With the above preparations, we can show \eqref{k} now. Write
\[
\mu \br{\left\{x: \abs{S_h^{\rho} \br{\sum_{i}e^{-k\rho(r_i)L }b_i} (x)}>\frac{\lambda}{3N} \right\}}
\lesssim \frac{1}{\lambda^2}\norm{\sum_{i}e^{-k\rho(r_i)L }b_i}_2^2
\]

By a duality argument, 
\begin{align*}
\norm{\sum_{i} e^{-k\rho(r_i)L} b_i}_2 
&= \sup_{\norm{\phi}_2=1} \int_{M} \abs{\sum_{i} e^{-k\rho(r_i)L} b_i} |\phi| d\mu
\leq \sup_{\norm{\phi}_2=1} \sum_{i} \sum_{j=1}^{\infty} \int_{C_j(B_i)} \abs{e^{-k\rho(r_i)L} b_i} |\phi| d\mu
\\ &=: \sup_{\norm{\phi}_2=1} \sum_{i} \sum_{j=1}^{\infty} A_{ij}.
\end{align*}

Applying Cauchy-Schwarz inequality, \eqref{off} and \eqref{D1}, we get
\begin{align*}
A_{ij} &\leq \norm{e^{-k\rho(r_i)L} b_i}_{L^2(C_j(B_i))} \norm{\phi}_{L^2(C_j(B_i))} 
\\ &\lesssim
2^{\frac{3j\nu}{2}} e^{-c 2^{j\tau(k \rho(r_i))}}\mu(B_i) \br{\frac{1}{\mu(B_i)}\int_{B_i} |b_i|^{p_0} d\mu}^{\frac{1}{p_0}} \inf_{y \in B_i}\br{\mathcal M\br{|\phi|^2}(y)}^{1/2}
\\ &\lesssim 
e^{-c 2^{j\tau(k \rho(r_i))}} \mu(B_i) \inf_{y \in B_i}\br{\mathcal M\br{|\phi|^2}(y)}^{1/2}.
\end{align*}
Here $\mathcal M$ denotes the Hardy-Littlewood maximal operator:
\[
\mathcal Mf(x)= \sup_{B\ni x} \frac 1{\mu(B)} \int_B |f(x)| d\mu(x),
\]
where $B$ ranges over all balls containing $x$. 

Then
\begin{align*}
\norm{\sum_{i} e^{-k\rho(r_i)L} b_i}_2 
&\lesssim
 \lambda \sup_{\norm{\phi}_2=1} \sum_{i}\sum_{j=1}^{\infty} e^{-c 2^{j\tau(k \rho(r_i))}} 
\mu(B_i) \inf_{y \in B_i}\br{M\br{|\phi|^2}(y)}^{1/2} 
\\ &\lesssim
 \lambda \sup_{\norm{\phi}_2=1}\int \sum_{i} \mathbbm 1_{B_i}(y) \br{\mathcal M\br{|\phi|^2}(y)}^{1/2} d\mu(y)
\\ &\lesssim
 \lambda \sup_{\norm{\phi}_2=1}\int_{\cup_{i}B_i}  \br{\mathcal M\br{|\phi|^2}(y)}^{1/2} d\mu(y)
\\ &\lesssim
 \lambda \mu^{1/2} \br{\cup_{i}B_i} \lesssim \lambda^{1-p_0/2} \br{\int |f|^{p_0} d\mu}^{1/2}.
\end{align*}
The third inequality is due to the finite overlap of the Calder\'on-Zygmund decomposition. In the last line, for the first inequality,
we use Kolmogorov's inequality (see for example \cite[page 91]{Gra08}).  

Therefore, we obtain
\begin{align}\label{S1}
\mu \br{\left\{x: \abs{S_h^{\rho} \br{\sum_{i}e^{-k\rho(r_i)L}b_i}(x)}>\frac{\lambda}{3N}\right\}} 
\lesssim \frac1{\lambda^{p_0}} \int |f|^{p_0} d\mu.
\end{align}

For the third term, we have
\begin{align*}
& \mu \br{\left\{x\in M: S_h^{\rho} \br{\sum_i \br{I-e^{-\rho (r_i)L }}^N b_i}(x)>\frac{\lambda}{3} \right\}}
\\ \leq&
\mu\br{\cup_j 4B_j}+\mu \br{\left\{x\in M\setminus \cup_j 4B_j: S_h^{\rho} \br{\sum_i \br{I-e^{-\rho(r_i)L}}^N b_i}(x)
>\frac{\lambda}{3} \right\}}.
\end{align*}

From the Calder\'on-Zygmund decomposition and doubling volume property, we get
\[
\mu\left(\cup_j 4B_j\right) \leq  \sum_j \mu(4B_j) \lesssim \sum_j \mu(B_j) \lesssim \frac{1}{\lambda^{p_0}}\Vert f\Vert_{p_0}.
\]

It remains to show that
\begin{align*}
\Lambda := \mu \left(\left\{x\in M\setminus \cup_j 4B_j: S_h^{\rho} \left(\sum_i \left(I-e^{-\rho(r_i)L } \right)^N b_i \right)(x) > 
\frac{\lambda}{3}\right\}\right)
\lesssim \frac{1}{\lambda^{p_0}} \int |f(x)|^{p_0} d\mu(x).
\end{align*}

As a consequence of the Chebichev inequality, $\Lambda$ is dominated by
\begin{align*}
& \frac{9}{\lambda^2 }\int_{ M\setminus \cup_j 4 B_j} 
\left (S_h^{\rho} \left(\sum_i  \left(I-e^{-\rho(r_i)L} \right)^N b_i \right)(x) \right )^2 d\mu (x)
\\ \leq &
\frac{9}{\lambda^2 }\int_{ M\setminus \bigcup_j 4 B_j} \iint_{\Gamma (x)}
\left (\sum_i \rho(t) L e^{-\rho(t)L }  \left(I-e^{-\rho(r_i)L} \right)^N b_i(y)\right )^2
\frac{d\mu(y)}{V(x,t)}\frac{dt}{t} d\mu (x)
\\ \leq & 
\frac{18}{\lambda^2 }\int_{ M\setminus \cup_j 4 B_j}\iint_{\Gamma (x)}
\left (\sum_i \mathbbm 1 _{2 B_i}(y) \rho(t) L e^{-\rho(t)L}  \left(I-e^{-\rho(r_i)L} \right)^N b_i(y)\right)^2
\frac{d\mu(y)}{V(x,t)}\frac{dt}{t}d\mu (x)
\\&+
\frac{18}{\lambda^2 }\int_{ M\setminus \cup_j 4 B_j}\iint_{\Gamma (x)} \br{\sum_i \mathbbm 1 _{M\setminus 2 B_i}(y) \rho(t) L e^{-
\rho(t)L } \br{I-e^{-\rho(r_i)L}}^N b_i(y)}^2 \frac{d\mu(y)}{V(x,t)}\frac{dt}{t}d\mu (x)
\\ =:&  \frac{18}{\lambda^2 }(\Lambda _{loc}+\Lambda _{glob}).
\end{align*}

For the estimate of $\Lambda_{loc}$. Due to the bounded overlap of $2B_i$, we can put the sum of $i$ out of the square up 
to a multiplicative constant. That is,
\begin{align*}
 \Lambda _{loc}
&\lesssim \sum_i \int_{ M\setminus \cup_j 4B_j} \int_{0}^{\infty}\int_{B(x,t)}
\br{\mathbbm 1_{2 B_i}(y) \rho(t) L e^{-\rho(t)L} \br{I-e^{-\rho(r_i)L}}^N b_i(y)}^2 \frac{d\mu(y)}{V(x,t)}\frac{dt}{t}d\mu(x)
\\ &\lesssim
 \sum_i \int_{ M\setminus \cup_j 4B_j}\int_{2 r_i}^{\infty} \int_{B(x,t)}
\br{\mathbbm 1_{2 B_i}(y) \rho(t)L e^{-\rho(t)L} \br{I-e^{-\rho(r_i)L}}^N b_i(y)}^2 \frac{d\mu(y)}{V(x,t)}\frac{dt}{t}d\mu(x)
\\  &\lesssim
 \sum_i \int_{2 r_i}^{\infty} \int_{M} \br{\int_{B(y,t)} \frac{d\mu(x)}{V(x,t)}}
\br{\mathbbm 1_{2 B_i}(y) \rho(t) L e^{-\rho(t)L}  \br{I-e^{-\rho(r_i)L}}^N b_i(y)}^2 d\mu (y)\frac{dt}{t}
\\  &\lesssim
 \sum_i \int_{2 r_i}^{\infty} \int_{2 B_i} \br{\rho(t)L e^{-\rho(t)L}  \br{I-e^{-\rho(r_i)L}}^N b_i(y)}^2 d\mu(y) \frac{dt}{t}.
\end{align*}
For the second inequality, note that for every $i$, $x \in M\setminus \cup_j 4 B_j$ means $x \notin 4B_i $. Then $y \in 2 B_i$
and $d(x,y)<t$ imply that $t\geq 2 r_i$. Thus the integral is zero for every $i$ if $0<t<2r_i$. We obtain the third inequality by using the Fubini  
theorem and \eqref{D2}.

Then by using \eqref{off}, it follows
\begin{align*}
\Lambda_{loc}
 &\lesssim
  \sum_i \int_{2 r_i}^{\infty} \int_{2 B_i} \br{\frac{\mu^{\frac1{p_0}-\frac1{2}}(B_i)}{V^{\frac1{p_0}-\frac1{2}}(y,t)} 
\frac{V^{\frac1{p_0}-\frac1{2}}(y,t)}{\mu^{\frac1{p_0}-\frac1{2}}(B_i)} \rho(t)L e^{-\rho(t)L}  
\br{I-e^{-\rho(r_i)L}}^N b_i(y)}^2 d\mu(y) \frac{dt}{t}
\\ &\lesssim
 \sum_i \int_{2 r_i}^{\infty} \int_{2 B_i} 
\br{\frac{V^{\frac1{p_0}-\frac1{2}}(y,4r_i)}{V^{\frac1{p_0}-\frac1{2}}(y,t)} 
\frac{V^{\frac1{p_0}-\frac1{2}}(y,t)}{\mu^{\frac1{p_0}-\frac1{2}}(B_i)} 
\rho(t)L e^{-\rho(t)L} \br{I-e^{-\rho(r_i)L}}^N b_i(y)}^2 d\mu(y) \frac{dt}{t}
\\  &\lesssim
\mu^{1-\frac{2}{p_0}}(B_i) \sum_i \int_{2 r_i}^{\infty} \br{\frac{4r_i}{t}}^{\nu' \br{\frac{2}{p_0}-1}}
\norm{V^{\frac1{p_0}-\frac1{2}}(\cdot,t) \rho(t)L e^{-\rho(t)L} \br{I-e^{-\rho(r_i)L}}^N b_i}_2^2 \frac{dt}{t}
\\ &\lesssim
\mu^{1-\frac{2}{p_0}}(B_i) \sum_i \norm{\br{I-e^{-\rho(r_i)L}}^N b_i}_{p_0}^2 
\\ &\lesssim
\mu^{1-\frac{2}{p_0}}(B_i) \sum_i \norm{b_i}_{p_0}^2  
\lesssim \lambda^2 \sum_i \mu(B_i) \lesssim \lambda^{2-p_0} \int |f|^{p_0} d\mu.
\end{align*}
For the second inequality, we use the reverse doubling property \eqref{revdb}. The third inequality follows from the $L^{p_0}-L^2$ 
boundedness of the operator $V^{\frac1{p_0}-\frac1{2}}(\cdot,t) \rho(t)L e^{-\rho(t)L}$ (see Lemma \ref{BK}). Then by using the $L^{p_0}$ boundedness 
of the heat semigroup, we get the fourth inequality. 

Now for the global part. We split the integral into annuli, that is,
\begin{align*}
\Lambda _{glob} &\leq 
\int_{ M}\iint_{\Gamma (x)}\left (\sum_i \mathbbm 1 _{M\setminus 2 B_i}(y) \rho(t) L e^{-\rho(t)L }
(I-e^{-\rho(r_i)L })^N b_i(y)\right )^2 \frac{d\mu(y)}{V(x,t)}\frac{dt}{t}d\mu (x)
\\  &\leq
\int_{0}^{\infty} \int_{ M} \int_{B(y,t)} \left (\sum_i \mathbbm 1_{M\setminus 2 B_i}(y) \rho(t) L e^{-\rho(t)L}
(I-e^{-\rho(r_i)L })^N b_i(y)\right )^2 \frac{d\mu(x)}{V(x,t)}d\mu (y)\frac{dt}{t}
\\  &\leq
\int_{0}^{\infty} \int_{ M}\left (\sum_i \mathbbm 1_{M\setminus 2 B_i}(y) \rho(t) L e^{-\rho(t)L}
(I-e^{-\rho(r_i)L })^N b_i(y)\right )^2 d\mu (y)\frac{dt}{t}.
\end{align*}

In order to estimate the above $L^2$ norm, we use an argument of dualization. Take the supremum of all functions
$h(y,t)\in L^2(M\times (0,\infty), \frac{d\mu dt}{t})$ with norm $1$, then
\begin{align*}
\Lambda_{glob}^{1/2} \leq& 
\br{\int_{0}^{\infty} \int_{M} \br{\sum_i \mathbbm 1_{M\setminus 2 B_i}(y) \rho(t)L e^{-\rho(t)L}(I-e^{-\rho(r_i)L })^N b_i(y)}^2 d\mu(y)
\frac{dt}{t}}^{1/2}
\\ =&
\sup_{h} \iint_{M\times (0,\infty)} \abs{\sum_i \mathbbm 1_{M\setminus 2 B_i}(y) \rho(t)L e^{-\rho(t)L} (I-e^{-\rho(r_i)L })^N b_i(y)} 
|h(y,t)| \frac{d\mu(y) dt}{t}
\\ \leq&
\sup_{h} \sum_i \sum_{j\geq 2}\int_{0}^{\infty}\int_{ C_j(B_i)} \abs{\rho(t)L e^{-\rho(t)L} (I-e^{-\rho(r_i)L })^N b_i(y)}
|h(y,t)| \frac{d\mu(y) dt}{t}
\\ \leq&
\sup_{h} \sum_i \sum_{j\geq 2} \br{\int_{0}^{\infty}\int_{ C_j(B_i)} \abs{\rho(t)L e^{-\rho(t)L} (I-e^{-\rho(r_i)L })^N b_i(y)}^2
\frac{d\mu (y) dt}{t}}^{1/2} 
\\ &\times \br{\int_{0}^{\infty}\int_{ C_j(B_i)} |h(y,t)|^2 \frac{d\mu (y) dt}{t}}^{1/2}.
\end{align*}

Denote $I_{ij}=\br{\int_{0}^{\infty}\int_{C_j(B_i)} \abs{\rho(t)L e^{-\rho(t)L} (I-e^{-\rho(r_i)L })^N b_i(y)}^2\frac{d\mu (y)dt}{t}}^{1/2}$. 

Let $H_{t,r}(\zeta)=\rho(t) \zeta e^{-\rho(t)\zeta}(1-e^{-\rho (r)\zeta})^N$. Then 
\begin{align}\label{I}
I_{ij}=\br{\int_{0}^{\infty} \Vert H_{t,r_i}(L)b_i \Vert_{L^2(C_j(B_i))}^2 \frac{dt}{t}}^{1/2}.
\end{align}

We will estimate $\Vert H_{t,r_i}(L)b_i \Vert_{L^2(C_j(B_i))}$ by functional calculus. The notation is mainly taken from 
\cite[Section 2.2]{Au07}.

For any fixed $t$ and $r$, then $H_{t,r}$ is a holomorphic function satisfying
\[
|H_{t,r}(\zeta)| \lesssim |\zeta|^{N+1} (1+|\zeta|)^{-2(N+1)},
\]
for all $\zeta \in{\Sigma }=\{z\in \mathbb C^{\ast }:|\arg z|<\xi \}$ with any $\xi \in (0,\pi/2)$.

Since $L$ is a nonnegative self-adjoint operator, or equivalently $L$ is a bisectorial operator of type $0$, we can express 
$H_{t,r}(L )$ by functional calculus. Let $0<\theta <\omega <\xi < \pi /2$, we have
\[
H_{t,r}(L )=\int_{\Gamma _{+}} e^{-zL }\eta _{+}(z)dz+\int_{\Gamma _{-}} e^{-zL }\eta _{-}(z)dz,
\]
where $\Gamma _{\pm }$ is the half-ray $\mathbb {R}^{+}e^{\pm i(\pi /2-\theta )}$ and
\begin{eqnarray*}
\eta _{\pm }(z) = \int_{\gamma _{\pm }} e^{\zeta z} H_{t,r}(\zeta ) d\zeta,\,\,\forall z\in \Gamma _{\pm },
\end{eqnarray*}
with $\gamma _{\pm }$ being the half-ray $\mathbb {R}^{\pm}e^{\pm i\omega }$.

Then for any $z\in \Gamma _{\pm }$,
\begin{align*}
|\eta _{\pm }(z)| &= \abs{\int_{\gamma _{\pm }} e^{\zeta z} \rho(t) \zeta e^{-\rho(t)\zeta} (1-e^{-\rho (r)\zeta})^N d\zeta}
\\ &\leq 
\int_{\gamma _{\pm}} |e^{\zeta z-\rho(t)\zeta}| \rho(t) |\zeta| |1-e^{-\rho(r)\zeta}|^N |d\zeta|
\\ &\leq 
\int_{\gamma _{\pm}} e^{-c|\zeta|(|z|+\rho(t))} \rho(t) |\zeta| |1-e^{-\rho(r)\zeta}|^N |d\zeta|
\\ &\lesssim 
 \int_{0}^{\infty} e^{-cs(|z|+\rho(t))} \rho(t) \rho^N(r) s^{N+1} ds
\leq  \frac{C\rho (t) \rho^N(r)}{(|z|+\rho (t))^{N+2}}.
\end{align*}
In the second inequality, the constant $c>0$ depends on $\theta$ and $\omega$. Indeed, 
$\Re(\zeta z)=|\zeta||z| \Re{e^{\pm i(\pi/2-\theta+\omega)}}$. Since $\theta<\omega$, then $\pi/2<\pi/2-\theta+\omega<\pi$ and 
$|e^{\zeta z}|=e^{-c_1 |\zeta| |z|}$ with $c_1=-\cos(\pi/2-\theta+\omega)$. Also it is obvious to see that 
$|e^{\rho(t)\zeta}|=e^{-c_2 \rho(t) |\zeta|}$. Thus the second inequality follows. In the third inequality, let 
$\zeta=s e^{\pm i \omega}$, we have $|d\zeta|=ds$. In addition, we dominate $|1-e^{-\rho(r)\zeta}|^N$ by $(\rho(r)\zeta)^N$.

We choose $\theta$ appropriately such that $|z|\sim \Re z$ for $z\in \Gamma _{\pm }$, then for any 
$j \geq 2$ fixed,
\begin{align*}
 \norm{H_{t,r_i}(L )b_i}_{L^2(C_j(B_i))}
&\lesssim
\br{\int_{\Gamma _{+}}+\int_{\Gamma _{-}}} \norm{e^{-\Re z L }b_i}_{L^2(C_j(B_i))}
\frac{\rho(t)}{(|z|+\rho (t))^2}\frac{\rho^N (r_i)}{(|z|+\rho (t))^N}|dz|
\\ &\lesssim
 \int_{0}^{\infty} \norm{e^{-sL }b_i}_{L^2(C_j(B_i))} \frac{\rho (t)\rho^N (r_i)}{(s+\rho (t))^{N+2}}ds.
\end{align*}

Applying Lemma \ref{BK}, then
\begin{align}\label{H}
\begin{split}
\left\Vert H_{t,r_i}(L )b_i\right\Vert_{L^2(C_j(B_i))}
 &\lesssim
\frac{2^{j\nu} \Vert b_i\Vert_{p_0}}{\mu^{\frac{1}{2}-\frac{1}{p_0}}(B_i)} 
\int_{0}^\infty e^{-c\left(\frac{2^j r_i}{\rho^{-1}(s)}\right)^{\tau(s)}} \frac{\rho (t)\rho^N (r_i)}{(s+\rho (t))^{N+2}} ds
\\ &\lesssim 
\frac{ 2^{j\nu} \Vert b_i\Vert_{p_0}}{\mu^{\frac{1}{2}-\frac{1}{p_0}}(B_i)} \br{\int_{0}^{\rho (t)}+\int_{\rho (t)}^{\infty }} 
e^{-c\br{\frac{2^j r_i}{\sigma(s)}}^{\tau(s)}} \frac{\rho (t)\rho^N (r_i)}{(s+\rho (t))^{N+2}} ds
\\ &=:  
\frac{ 2^{j\nu} \Vert b_i\Vert_{p_0}}{\mu^{\frac{1}{2}-\frac{1}{p_0}}(B_i)} (H_1(t,r_i,j)+H_2(t,r_i,j)).
\end{split}
\end{align}
In the second and the third lines, $\tau(s)$ is originally defined in \eqref{off}. In fact, it should be $\tau(\rho^{-1}(s))$. Since 
$\rho^{-1}(s)$ and $s$ are unanimously larger or smaller than one, we always have $\tau(s)=\tau(\rho^{-1}(s))$.

Hence, by Minkowski inequality, we get from \eqref{I} and \eqref{H} that
\begin{align}\label{I_ij}
I_{ij} \lesssim
\frac{ 2^{j\nu} \Vert b_i\Vert_{p_0}}{\mu^{\frac{1}{2}-\frac{1}{p_0}}(B_i)} 
\br{\br{\int_{0}^{\infty} H_1^2(t,r_i,j) \frac{dt}{t}}^{1/2} + \br{\int_{0}^{\infty} H_2^2(t,r_i,j) \frac{dt}{t}}^{1/2}}.
\end{align}

It remains to estimate the two integrals $\int_{0}^{\infty} H_1^2(t,r_i,j) \frac{dt}{t}$ and $\int_{0}^{\infty} H_2^2(t,r_i,j) \frac{dt}{t}$. We claim that
\begin{align}\label{H1}
\int_{0}^{\infty} H_1^2(t,r_i,j) \frac{dt}{t},\,\int_{0}^{\infty }H_2^2(t,r_i,j)\frac{dt}{t} \lesssim 2^{-2\beta_1N j}.
\end{align}

Estimate first $\int_{0}^{\infty} H_1^2(t,r_i,j) \frac{dt}{t}$. Since $\frac{\rho (t)\rho^N (r_i)}{(s+\rho (t))^{N+2}}\leq \frac{\rho^N (r_i)}{\rho(t)^{N+1}} $, we obtain
\[
H_1(t,r_i,j) \leq \int_{0}^{\rho(t)} e^{-c\br{\frac{2^{j} r_i}{\sigma(s)}}^{\beta_2/(\beta_2-1)}} \frac{\rho^N(r_i)}{\rho^{N+1}(t)}ds
\lesssim 
e^{-c\br{\frac{2^{j} r_i}{t}}^{\beta_2/(\beta_2-1)}} \frac{\rho^N (r_i)}{\rho^N(t)}.
\]
It follows that
\begin{align*}
\int_{0}^{\infty }H_1^2(t,r_i,j)\frac{dt}{t} 
&\lesssim 
\int_{0}^{\infty}e^{-2c\br{\frac{2^{j} r_i}{t}}^{\beta_2/(\beta_2-1)}} \frac{\rho^{2N} (r_i)}{\rho^{2N}(t)} \frac{dt}{t} 
\\ &\lesssim 
 \int_{0}^{2^j r_i } \br{\frac{t}{2^{j} r_i}}^{c} \frac{\rho^{2N} (r_i)}{\rho^{2N}(t)} \frac{dt}{t}
+  \int_{2^j r_i}^{\infty} \frac{\rho^{2N} (r_i)}{\rho^{2N}(t)} \frac{dt}{t}
\\ &\lesssim \frac{\rho^{2N} (r_i)}{\rho^{2N}(2^j r_i)}
\lesssim 2^{-2\beta_1N j} 
\end{align*}
In the first inequality, we dominate the exponential term by polynomial 
one for the first integral, where $c$ in the second line is chosen to be larger than $2\beta_2 N$. 

\medskip

Now estimate $\int_{0}^{\infty }H_2^2(t,r_i,j)\frac{dt}{t}$. Write $\frac{\rho (t)\rho^N (r_i)}{(s+\rho (t))^{N+2}} \leq \frac{\rho (t) \rho^N(r_i)}{s^{N+2}}$. On the one 
hand,
\begin{align}\label{h2}
H_2(t,r_i,j) =\int_{\rho (t)}^{\infty }
e^{-c\left(\frac{2^j r_i}{\sigma(s)}\right)^{\tau(s)}} \frac{\rho (t)\rho^N (r_i)}{(s+\rho (t))^{N+2}} ds
\leq \int_{\rho(t)}^{\infty } \frac{\rho(t) \rho^N(r_i)}{s^{N+2}} ds=C\frac{\rho^N (r_i)}{\rho^N (t)}.
\end{align}
On the other hand, we also have 
\begin{align}\label{h2'}
H_2(t,r_i,j)  \lesssim 2^{-\beta_1 Nj} \frac{\rho (t)}{\rho (2^jr_i)}.
\end{align}
In fact,  
\begin{align*}
H_2(t,r_i,j) &\leq  
\int_{\rho(t)}^{\infty } e^{-c\br{\frac{2^{j} r_i}{\sigma(s)}}^{\beta_2/(\beta_2-1)}} \frac{\rho(t)\rho^N(r_i)}{s^{N+1}}\frac{ds}{s}
\\ &\lesssim
 2^{-\beta_1 Nj} \frac{\rho (t)}{\rho (2^jr_i)} \int_{\rho(t)}^{\infty} e^{-c\br{\frac{2^{j} r_i}{\sigma(s)}}^{\beta_2/(\beta_2-1)}} \frac{\rho^{N+1}(2^j r_i)}{s^{N+1}}\frac{ds}{s}
\\ &\lesssim
 2^{-\beta_1 Nj} \frac{\rho (t)}{\rho (2^jr_i)}.
\end{align*}
%Combining (\ref{h2}) and (\ref{h2'}), we get
%\begin{align}
%H_2(t,r_i,j) \leq   C \inf \left\{ \frac{\rho^N (r_i)}{\rho^N (t)}, 2^{-\beta_1 Nj} \frac{\rho (t)}{\rho (2^j r_i)}\right\}.
%\end{align}
Now we split the integral into two parts in the same way and control them by using \eqref{h2} and \eqref{h2'} seperately. Then
\begin{align*}
\int_{0}^{\infty } H_2^2(t,r_i,j) \frac{dt}{t} &\lesssim
\int_{0}^{2^j r_i } 2^{-2\beta_1 Nj} \frac{\rho^2(t)}{\rho^2(2^j r_i)}\frac{dt}{t} 
+\int_{2^j r_i }^{\infty } \frac{\rho^{2N} (r_i)}{\rho^{2N} (t)} \frac{dt}{t}
\\ &\lesssim
2^{-2\beta_1 Nj}.
\end{align*}

\medskip

Therefore, it follows from \eqref{I_ij} and \eqref{H1} that
\begin{align}\label{Iij}
I_{ij} \lesssim \frac{ \mu^{1/2}(2^{j}B_i) \Vert b_i\Vert_{p_0}}{\mu^{1/p_0}(B_i)} 2^{-\beta_1 N j}.
\end{align}

Now for the integral $\left(\int_{0}^{\infty}\int_{ C_j(B_i)} |h(y,t)|^2 \frac{d\mu (y) dt}{t}\right)^{1/2}$. Take 
$\tilde h(y)=\int_{0}^{\infty}|h(y,t)|^2 \frac{dt}{t}$, then
\begin{align}\label{h}
\br{\int_{0}^{\infty}\int_{ C_j(B_i)} |h(y,t)|^2 \frac{d\mu (y) dt}{t}}^{1/2}
\leq \mu^{1/2}(2^{j+1}B_i) \inf_{z \in B_i} \mathcal M^{1/2} \tilde h (z),
\end{align}
where $\mathcal M$ is the Hardy-Littlewood maximal function.

Following the route for the proof of \eqref{k}, we get from  \eqref{Iij} and \eqref{h} that
\begin{align*}
\Lambda_{glob}^{1/2} 
&\lesssim
\sup_{h} \sum_i\sum_{j\geq 2} \frac{ 2^{j\nu} \Vert b_i\Vert_{p_0}}{\mu^{\frac{1}{2}-\frac{1}{p_0}}(B_i)} 
2^{-\beta_1 N j} \mu^{1/2}(2^{j+1}B_i)\inf_{z \in B_i} \mathcal M^{1/2}\tilde h (z)
\\ &\lesssim
\lambda \sup_{h} \int_M \sum_i \mathbbm 1_{B_i}(y) \mathcal M^{1/2}\tilde h(y) d\mu(y)
\\ &\lesssim
\lambda \sup_{h} \int_{\cup_i B_i} \mathcal M^{1/2} \tilde h (y) d\mu(y)
\\ &\lesssim 
\lambda \mu(\cup_i B_i)^{1/2} 
\lesssim \lambda^{1-p_0/2} \int |f|^{p_0} d\mu.
\end{align*}
Here the supremum is taken over all the functions $h$ with $\norm{h}_{L^2\br{\frac{d\mu dt}{t}}}=1$. Since $N>2\nu/\beta_1$, 
the sum $\sum_{j\geq 2} 2^{-\beta_1N j+3\nu j/2}$ converges and we get the second inequality. The fourth one is a result of Kolmogorov's 
inequality.

Thus we have shown $\Lambda _{glob} \lesssim \lambda^{2-p_0} \int |f|^{p_0} d\mu$.
\end{proof}

\medskip
%%%%%%%%%%%%%%%%%%%%%%%%%%%%%%%%%%%%%%%%%%%%%%%%%%%%%%%%%%%%%%%

\subsection{Counterexamples to $H_{L,\,S_h}^p(M) = L^p(M)$}

Before moving forward to the proof of Theorem \ref{noequiv}, let us recall the following two theorems about the 
Sobolev inequality and the Green operator.
\begin{thm}[\cite{Co90}]\label{hkSob}
Let $(M,\mu)$ be a $\sigma-$finite measure space. Let $T_t$ be a semigroup on $L^s$, $1\leq s\leq \infty$, with 
infinitesimal generator $-L$. Assume that $T_t$ is equicontinuous on $L^1$ and $L^{\infty}$. Then the following two 
conditions are equivalent:
\begin{enumerate}
\item There exists $C>0$ such that $\norm{T_t}_{1\rightarrow \infty} \leq C t^{-D/2}$, $\forall t \geq 1$.

\item $T_1$ is from $L^1$ to $L^{\infty}$ and for $q>1$, $\exists C$ such that
\begin{align}\label{sob}
\norm{f}_{p} \leq C \br{\norm{L^{\alpha/2} f}_{q}+\norm{L^{\alpha/2} f}_{p}},\,\,f\in \mathcal D (L^{\alpha/2}) ,
\end{align}
where $0<\alpha q<D$ and $\frac{1}{p}=\frac{1}{q}-\frac{\alpha}{D}$.
\end{enumerate}
\end{thm}

\begin{thm}\label{Gr}
Let $M$ be a complete non-compact manifold. 
%without boundary. 
Then there exists a Green's function $G(x,y)$ which is smooth on $(M\times M)\backslash D$ satisfying
\begin{align}
\Delta_x \int_M G(x,y) f(y) d\mu(y)=f(x),\,\, \forall f\in \mathcal C_0^{\infty}(M).
\end{align}
\end{thm}
For a proof, see for example \cite{Li12}.

\medskip
We also observe that
\begin{lem}\label{LB}
Let $M$ be a Riemannian manifold satisfying the polynomial volume growth \eqref{d} and the two-sided sub-Gaussian heat kernel estimate $(HK_{2,m})$. Let $B$ be an arbitrary ball with radius $r\geq 4$. Then there exists a constant $c>0$ depending on $d$ and 
$m$ such that for all $t$ with $r^m/2 \leq t \leq r^m$, 
\[
\int_B p_t(x,y) d\mu(y) \geq c,\,\,\forall x \in B.
\]
\end{lem}
\begin{proof}
Note that for any $x,y \in B$, we have $t \geq r^m/2 \geq 2r \geq d(x,y)$. Then ($H\!K_{2,m}$) yields
\begin{align*}
\int_B p_t(x,y) d\mu(y) 
&\geq \int_B \frac{c}{t^{d/m}} \exp\br{-C\br{\frac{d^m(x,y)}{t}}^{1/(m-1)}} d\mu(y)
\\ &\geq \frac{c \mu(B)}{t^{d/m}} \exp\br{-C\br{\frac{r^m}{t}}^{1/(m-1)}} 
\geq c.
\end{align*}
\end{proof}

\medskip
\begin{proof}[Proof of Theorem \ref{noequiv}]

Let $\phi_n \in \mathcal C_0^{\infty}(M)$ be a cut-off function as follows: $0 \leq \phi_n \leq 1$ and for some $x_0 \in M$,
\[
\phi_n(x)=\left\{ \begin{aligned}
&1,&x\in B(x_0,n),\\
&0,&x\in M\backslash B(x_0,2n).
\end{aligned} \right.
\]
For simplicity, we denote $B(x_0,n)$ by $B_n$.

Taking $f_n=G\phi_n$, Theorem \ref{Gr} says that $\Delta f_n=\phi_n$.

On the one hand, we apply Theorem \ref{hkSob} by choosing $T_t=e^{-t\Delta}$. Indeed, $e^{-t\Delta}$ is Markov hence bounded 
on $L^p$, equicontinuous on $L^1, L^{\infty}$ and satisfies 
\[
\norm{e^{-t\Delta}}_{1\rightarrow \infty} =\sup_{x,y \in M} p_t(x,y) \leq C t^{-D/2},
\]
where $D=2d/m>2$. Then taking $\alpha=2$ and $p>\frac{D}{D-2}$, it follows that
\begin{equation*}
\norm{f_n}_p \leq C \br{\norm{\Delta f_n}_{q}+\norm{\Delta f_n}_p},
\end{equation*}
where $\frac{1}{p}=\frac{1}{q}-\frac{\alpha}{D}$, that is, $q=\frac{Dp}{D+2p}=\frac{dp}{d+mp}$. 

Using the fact that $\Delta f_n=\phi_n$ and $\phi_n \leq \mathbbm 1_{B(x_0,2n)}$, we get
\begin{align}\label{sobeq}
\begin{split}
\norm{f_n}_p 
&\lesssim \br{\norm{\phi_n}_{\frac{dp}{d+mp}}+\norm{\phi_n}_p} 
\lesssim\br{V^{\frac{d+mp}{dp}}(x_0,2n)+V^{\frac{1}{p}}(x_0,2n)} 
\\ &\lesssim \br{n^{m+d/p}+n^{d/p}} 
\lesssim n^{m+d/p}.
\end{split}
\end{align}
In particular, $\norm{f_n}_2 \lesssim n^{m+d/2}$.

On the other hand, 
\begin{align*}
\norm{S_h f_n}_p^p 
&= \int_M \br{\iint_{\Gamma(x)} \abs{t^2 \Delta e^{-t^2 \Delta} f_n(y)}^2 \frac{d\mu(y)}{V(x,t)} \frac{dt}{t} }^{p/2} d\mu(x)
\\ &= 
\int_M \br{\iint_{\Gamma(x)} \abs{t^2 e^{-t^2 \Delta} \phi_n(y)}^2 \frac{d\mu(y)}{V(x,t)} \frac{dt}{t} }^{p/2} d\mu(x).
\end{align*}
Since $\phi_n \geq \mathbbm 1_{B_{n}} \geq 0$, it follows from the Markovian property of the heat semigroup that
\[
\norm{S_h f_n}_p^p \geq
\int_M \br{\iint_{\Gamma(x)} \abs{t^2 e^{-t^2 L} \mathbbm 1_{B_{n}} (y)}^2 \frac{d\mu(y)}{V(x,t)} \frac{dt}{t} }^{p/2} d\mu(x).
\]
By using Lemma \ref{LB}, it holds that $e^{-t^2L} \mathbbm 1_{B_{n/2}} \geq c$ if $\frac{n^{m/2}}{2} \leq t \leq n^{m/2}$. 
Then we get
\[
\norm{S_h f_n}_p^p \gtrsim
\int_{B\br{x_0, \frac{n^{m/2}}{4}}} \br{\int_{\frac{n^{m/2}}{2}}^{n^{m/2}} \int_{B(x,t) \cap B_{n/2}} \frac{t^3}{V(x,t)} d
\mu(y) dt }^{p/2} d\mu(x).
\]
Observe also that, for $t>\frac{n^{m/2}}{2}$ and $x\in B\br{x_0, \frac{n^{m/2}}{4}}$, we have $B_{n} \subset B(x,t)$ as 
long as $n$ is large enough. Then the volume growth \eqref{d} gives us a lower bound in terms of $n$. That is,
\[
\norm{S_h f_n}_p^p \gtrsim 
\int_{B\br{x_0, \frac{n^{m/2}}{4}}} \br{\int_{\frac{n^{m/2}}{2}}^{n^{m/2}} \frac{ \mu(B_n) t^3}{V(x,n^{m/2})} dt}^{p/2} 
d\mu(x)
\gtrsim n^{\frac{md}{2}(1-\frac{p}{2})} n^{mp+dp/2}.
\]

Comparing the upper bound of $\norm{f_n}_p$ in \eqref{sobeq} for $p>\frac{D}{D-2}$, we obtain
\begin{align}\label{noeq}
\norm{S_h f_n}_p \gtrsim n^{\frac{md}{2}(\frac{1}{p}-\frac{1}{2})+m+\frac{d}{2}} 
\gtrsim n^{d\br{\frac{m}{2}-1} \br{\frac{1}{p}-\frac{1}{2}}} \norm{f_n}_p,
\end{align}
where $p>\frac{D}{D-2}$.

Assume $D>4$, i.e. $m<d/2$, we have $\frac{D}{D-2}<2$. Then for $\frac{D}{D-2}<p<2$, since $m>2$, 
\[
n^{d\br{\frac{m}{2}-1} \br{\frac{1}{p}-\frac{1}{2}}} \rightarrow \infty \text{   as  } n\rightarrow \infty.
\]
Thus (\ref{noeq}) implies that $L^p \subset H_{S_h}^p$ is not true for $p\in \br{\frac{D}{D-2},2}$, i.e. $p \in \br{\frac{d}{d-
m},2}$, where $2<m < d/2$.

Our conclusion is: for any fixed $p \in \br{\frac{d}{d-m},2}$, according to \eqref{sobeq} and \eqref{noeq}, there exists a family of 
functions $\left\{g_n=\frac{f_n}{n^{m+d/p}}\right\}_{n\geq 1}$ such that $\norm{g_n}_p \leq C$, 
$\norm{g_n}_2 \leq n^{\frac{d}{2}-\frac{d}{p}} \to 0$ and $\norm{S_h g_n}_p \geq n^{d(\frac{m}{2}-1)(\frac{1}{p}-\frac{1}{2})}\to+\infty$ 
as $n$ goes to infinity. Therefore $S_h$ is not $L^p$ bounded for $p \in \br{\frac{d}{d-m},2}$ and the inclusion 
$L^p \subset H_{S_h^{m'}}^p$ doesn't hold for $p \in \br{\frac{d}{d-m},2}$.
\end{proof}

\medskip
More generally, a slight adaption of Theorem \ref{noequiv} plus Theorem \ref{equihl} yields the following result. 
\begin{cor}\label{neq-gen}
Let $M$ be a Riemannian manifold satisfying \eqref{d} and ($H\!K_{2,m}$) as above. Let $p\in \br{\frac{d}{d-m}, 2}$.
Then for any $0<m'\leq m$, $L^p(M) =H_{S_h^{m'}}^p(M)$ if and only if $m'=m$.
\end{cor}
\begin{proof}
If $m'=m$, Theorem \ref{equihl} says that $L^p \subset H_{S_h^{m}}^p$.

Conversely, by doing a slight adjustment for the above proof, we can show that $L^p \subset H_{S_h^{m'}}^p$ is false for 
$p \in \br{\frac{d}{d-m},2}$, where $2<m < d/2$ and $m'<m$.
\end{proof}
\bigskip
%%%%%%%%%%%%%%%%%%%%%%%%%%%%%%%%%%%%%%%%%%%%%%%%%%%%%%%%%%%%%%%%

\section{The $H^1-L^1$ boundedness of Riesz transforms on fractal manifolds}
This section is devoted to an application of the Hardy space theory we introduced above. 

Let $(M,d,\mu)$ be a Riemannian manifold satisfying the doubling volume property ($D$) and the sub-Gaussian estimate 
$(U\!E_{2,m})$. Note that we could as well consider a metric measure Dirichlet space which admits a ``carr\'e du champ'' (see, for example,
\cite{BE85,GSC11}).

Recall that the Riesz transform $\nabla \Delta^{-1/2}$ is of weak type $(1,1)$ on $M$:
\begin{thm}[\cite{CCFR15}]
Let $M$ be a manifold satisfying the doubling volume property \eqref{doubling} and the heat kernel estimate $(U\!E_{2,m})$, 
$m>2$. Then, the Riesz transform is weak $(1,1)$ bounded and bounded on $L^p$ for $1<p\leq 2$.
\end{thm}
 
The  proof depends on the following integrated estimate for the gradient of the heat kernel.
\begin{lem}[\cite{CCFR15}] \label{EstimateKernelCor}
Let $M$ be as above.  
Then for all $y\in M$, all $r,t>0$,
\begin{align}\label{intenew}
\int_{M\setminus B(y,r)} \left\vert \nabla_x h_t(x,y)\right\vert \,d\mu(x)\lesssim \frac 1{\sqrt{t}} \exp\left(-c\left(\frac{\rho(r)}t\right)^\frac{1}{m-1}\right),
\end{align}
where $\rho$ is defined in \eqref{rho}. 
\end{lem}

Our aim here is to prove Theorem \ref{thm2}. More specifically, we will show that the Riesz transform is $H_{\Delta,m,mol}^1(M)-L^1(M)$ bounded. Due to Theorem \ref{H1equiv}, it is $H_{\Delta,m}^1(M)-L^1(M)$ bounded. The method we use is similar as in \cite[Theorem 3.2]{HM09}. Note that the pointwise assumption \eqref{ue} simplifies the proof below.

Note first the following lemma, which is crucial in our proof.
\begin{lem}\label{time derivative}
Let $M$ be as above and let $p\in (1,2)$. Then for any $E, F\subset M$ and for any $n\in \N$, we have
\begin{equation}\label{gradient-time}
\norm{\left|\nabla \Delta^n e^{-t\Delta} f\right|}_{L^p(F)} \lesssim
\left\{ \begin{aligned}
         &\frac{1}{t^{n+1/2}} e^{-c\frac{d^2(E,F)}{t}} \norm{f}_{L^p(E)},&0<t<1, \\
         &\frac{1}{t^{n+1/2}} e^{-c\br{\frac{d^m(E,F)}{t}}^{1/(m-1)}} \norm{f}_{L^p(E)},&t\geq 1;
         \end{aligned}\right.
\end{equation}
where $f\in L^p(M)$ is supported in $E$.
Consequently, 
\begin{align}\label{DG2}
\norm{\left|\nabla \Delta^n e^{-t\Delta} f\right|}_{L^p(F)} \lesssim
\frac{1}{t^{n+1/2}} e^{-c\br{\frac{\rho(d(E,F))}{t}}^{1/(m-1)}} \norm{f}_{L^p(E)}.
\end{align}

\end{lem}
\begin{rem}
To prove the lemma, it is enough to show that the following two estimates:
\begin{equation*}
\norm{\left|\nabla e^{-t\Delta}f\right|}_{L^p(F)} \lesssim 
\left\{ \begin{aligned}
         & e^{-c\frac{d^2(E,F)}{t}} \norm{f}_{L^p(E)},&0<t<1, \\
         & e^{-c\br{\frac{d^m(E,F)}{t}}^{1/(m-1)}} \norm{f}_{L^p(E)},&t\geq 1,
         \end{aligned}\right.
\end{equation*}
and
\begin{equation*}
\norm{(t\Delta)^n e^{-t\Delta} f}_{L^p(F)} \lesssim 
\left\{ \begin{aligned}
         & e^{-c\frac{d^2(E,F)}{t}} \norm{f}_{L^p(E)},&0<t<1, \\
         & e^{-c\br{\frac{d^m(E,F)}{t}}^{1/(m-1)}} \norm{f}_{L^p(E)},&t\geq 1.
         \end{aligned}\right.
\end{equation*}
Then \eqref{gradient-time} follows by adapting the proof of \cite[Lemma 2.3]{HM03}. Note that the first estimate can be obtained by using Stein's approach, similarly as the proof of Lemma \ref{EstimateKernelCor}. The second estimate is a direct consequence of \eqref{ue} and the analyticity of the heat semigroup (see \cite{Fe15} for its discrete analogue). We omit the details of the proof here.
\end{rem}

\begin {rem}
Note that \eqref{gradient-time} implies \eqref{DG2} (see \cite[Corollary 2.4]{CCFR15}), which may simplify the calculation in the subsequent proofs.
\end{rem}

%\begin{thm}
%\label{h1l1}
%Let $M$ be a manifold satisfying the doubling volume property \eqref{doubling} and the heat kernel estimate $(U\!E_{2,m})$, 
%$m> 2$. Then the Riesz transform is bounded from $H_{\Delta,m,\mol}^1(M)$ into $L^1(M)$.
%\end{thm}
\medskip

\begin{proof}[Proof of Theorem \ref{thm2}]

Denote by $T:=\nabla \Delta^{-1/2}$. It suffices to show that, for any $(1,2,\varepsilon)-$molecule $a$ associated to a function $b$ and a ball $B$ with radius $r_B$, there exists a constant $C$ such that $\Vert Ta\Vert_{L^1(M)}\leq C$.

Write 
\begin{align}\label{divide}
T a=T e^{-\rho(r_B)\Delta }a+T \left(I-e^{-\rho(r_B)\Delta }\right)a.
\end{align}
Then
\[
\norm{Ta}_{L^1(M)} \leq 
\norm{T \left(I-e^{-\rho(r_B)\Delta }\right)a}_{L^1(M)}+\norm{T e^{-\rho(r_B)\Delta }a}_{L^1(M)}
=: I+II.
\]

We first estimate  $I$. It holds that
\begin{align*}
I &\leq \sum_{i\geq 1} \norm{T \left(I-e^{-\rho(r_B)\Delta }\right)\mathbbm 1_{C_i(B)} a}_{L^1(M)}
\\ &\leq
\sum_{i\geq 1} \br{\norm{T \left(I-e^{-\rho(r_B)\Delta }\right)\mathbbm 1_{C_i(B)} a}_{L^1(M\backslash 2^{i+2}B)}+\norm{T \left(I-e^{-\rho(r_B)\Delta }\right)\mathbbm 1_{C_i(B)} a}_{L^1(2^{i+2}B)}}
\end{align*}
Using the Cauchy-Schwarz inequality and the $L^2$ boundedness of $T$ and $e^{-\rho(r_B)\Delta }$, it follows that
\begin{align}\label{I11}
\norm{T \left(I-e^{-\rho(r_B)\Delta }\right)\mathbbm 1_{C_i(B)} a}_{L^1(2^{i+2}B)}
 \lesssim V(2^{i+2}B) \norm{a}_{L^2(C_i(B))} \lesssim 2^{-i\varepsilon}.
\end{align}
Now we claim:
\begin{align}\label{I12}
\norm{T \left(I-e^{-\rho(r_B)\Delta }\right)\mathbbm 1_{C_i(B)} a}_{L^1(M\backslash 2^{i+2}B)}
 \lesssim 2^{-i\varepsilon}.
\end{align}
Combining \eqref{I11} and \eqref{I12}, we obtain that $I$ is bounded.

In order to prove $\eqref{I12}$, we adapt the trick in \cite{CCFR15}. For the sake of completeness, we write it down. First note that the spectral theorem gives us $\Delta ^{-1/2}f=c\int_0^\infty e^{-s\Delta}f\frac{ds}{\sqrt s}$. 
Therefore, 
\[\begin{split} 
\Delta^{-1/2} (I-e^{-t\Delta}) a &=c\int_0^\infty (e^{-s\Delta} - e^{-(s+\rho(r_B))\Delta})a\frac{ds}{\sqrt s} \\
& = c\int_{0}^\infty \left( \frac{1}{\sqrt s} - \frac{\chi_{\{s>\rho(r_B)\}}}{\sqrt{s-\rho(r_B)}}\right) e^{-s\Delta} a\, ds.
\end{split} \]

Set 
$$
k_{\rho(r_B)}(x,y) = \int_0^\infty \abs{\frac{1}{\sqrt s} - \frac{\chi_{\{s>\rho(r_B)\}}}{\sqrt{s-\rho(r_B)}}} |\nabla_x  h_s(x,y)| ds.
$$
Then
\begin{align*}
\norm{T \br{I-e^{-\rho(r_B)\Delta }}\mathbbm 1_{C_i(B)} a}_{L^1(M\backslash 2^{i+2}B)}
&\lesssim  
\int_{M\backslash 2^{i+2}B)}\int_{C_i(B)}k_{\rho(r_B)}(x,y)|a(y)|d\mu(y)d\mu(x)
\\&\lesssim 
\int_{C_i(B)} |a(y)|\int_{d(x,y)\geq 2^i r}k_{\rho(r_B)}(x,y)d\mu(x)d\mu(y).
\end{align*}
It remains to show that $\int_{d(x,y)\geq 2^i r}k_{\rho(r_B)}(x,y)d\mu(x)$ converges uniformly. Indeed, Lemma \ref{EstimateKernelCor} yields
\begin{align*}
\int_{d(x,y)\geq 2^i r}k_{\rho(r_B)}(x,y)d\mu(x) 
&=
\int_0^\infty \abs{\frac{1}{\sqrt s} - \frac{\chi_{\{s>\rho(r_B)\}}}{\sqrt{s-\rho(r_B)}}}\int_{d(x,y)\geq 2^i r} |\nabla_x  h_s(x,y)| d\mu(x) ds
\\ &\lesssim
\int_0^\infty \abs{\frac{1}{\sqrt s} - \frac{\chi_{\{s>\rho(r_B)\}}}{\sqrt{s-\rho(r_B)}}} \frac 1{\sqrt{s}} \exp\br{-c\br{\frac{\rho(2^ir)}s}^\frac{1}{m-1}}ds
\\ &\lesssim 1.
\end{align*}

Now turn to estimate $II$. We have 
\begin{align*}
II &= \norm{c\int_0^\infty \nabla e^{-(s+\rho(r_B))\Delta}a\frac{ds}{\sqrt s}}_{L^1(M)}
\\ &\lesssim 
\int_0^{\rho (r_B)} \norm{\left|\nabla e^{-(s+\rho(r_B))\Delta} a\right|}_{L^1(M)} \frac{ds}{\sqrt s}
+
\int_{\rho (r_B)}^{\infty} \norm{\left|\nabla e^{-(s+\rho(r_B))\Delta} \Delta^K b\right|}_{L^1(M)} \frac{ds}{\sqrt s}
\\ &=: 
II_1+II_2.
\end{align*}

We estimate $II_1$ as follows:
\begin{equation*}
II_1 \leq 
\sum_{i\geq 1} \int_0^{\rho (r_B)} \br{\norm{\left|\nabla e^{-(s+\rho(r_B))\Delta} \mathbbm 1_{C_i(B)}a\right|}_{L^1(2^{i+2}B)} 
+ \norm{\left|\nabla e^{-(s+\rho(r_B))\Delta} \mathbbm 1_{C_i(B)}a\right|}_{L^1(M\backslash 2^{i+2}B)}} \frac{ds}{\sqrt s}.
\end{equation*}
Estimate the first term inside the sum by Cauchy-Schwarz and the fact that $\norm{e^{-t\Delta}}_{2\to 2} \lesssim \frac{1}{\sqrt t}$. Then
\begin{equation*}
\begin{split}
\int_0^{\rho (r_B)} \norm{\left|\nabla e^{-(s+\rho(r_B))\Delta} \mathbbm 1_{C_i(B)}a\right|}_{L^1(2^{i+2}B)} \frac{ds}{\sqrt s}
&\lesssim 
\int_0^{\rho (r_B)} V^{1/2}(2^{i+2}B) \norm{a}_{L^2(C_i(B))} \frac{ds}{\sqrt {s+\rho(r_B)} \sqrt s}
\\ &\lesssim 
2^{-i\varepsilon} \int_0^{\rho (r_B)}  \frac{ds}{\rho(r_B) \sqrt s} 
\\ &\lesssim 
2^{-i\varepsilon}
\end{split}
\end{equation*}
For the second term inside the sum, we use Lemma \ref{EstimateKernelCor} again. Then
\begin{equation*}
\begin{split}
&\int_0^{\rho (r_B)} \norm{\left|\nabla e^{-(s+\rho(r_B))\Delta} \mathbbm 1_{C_i(B)}a\right|}_{L^1(M\backslash 2^{i+2}B)} \frac{ds}{\sqrt s}
\\ &\lesssim 
\int_0^{\rho (r_B)} \int_{M\backslash 2^{i+2}B)} \int_{C_i(B)} \abs{\nabla p_{s+\rho(r_B)}(x,y) a(y)} d\mu(y)d\mu(x) \frac{ds}{\sqrt s}
\\ &\lesssim 
\int_0^{\rho (r_B)} \int_{C_i(B)} \int_{d(x,y)\geq 2^{i+1}B} \abs{\nabla p_{s+\rho(r_B)}(x,y)} d\mu(x) \abs{a(y)} d\mu(y) \frac{ds}{\sqrt s}
\\ &\lesssim 
\norm{a}_{L^1(C_i(B))} \int_0^{\rho (r_B)} \frac{ds}{\sqrt {s+\rho(r_B)} \sqrt s}
\\ &\lesssim 
2^{-i\varepsilon}.
\end{split}
\end{equation*}

It remains to estimate $II_2$. Using the same method as for $II_1$, then
\begin{equation*}
II_2 \leq 
\sum_{i\geq 1} \int_{\rho (r_B)}^{\infty} \br{\norm{\left|\nabla e^{-(s+\rho(r_B))\Delta}\Delta^K \mathbbm 1_{C_i(B)}b\right|}_{L^1(2^{i+2}B)} 
+ \norm{\left|\nabla e^{-(s+\rho(r_B))\Delta} \Delta^K\mathbbm 1_{C_i(B)}b\right|}_{L^1(M\backslash 2^{i+2}B)}} \frac{ds}{\sqrt s}.
\end{equation*}
For the first term inside the sum, we estimate by using Cauchy-Schwartz inequality and the spectral theory. Then
\begin{equation*}
\begin{split}
\int_{\rho (r_B)}^{\infty} \norm{\left|\nabla \Delta^K e^{-(s+\rho(r_B))\Delta} \mathbbm 1_{C_i(B)}b\right|}_{L^1(2^{i+2}B)} \frac{ds}{\sqrt s}
&\lesssim 
\int_{\rho (r_B)}^{\infty} \mu^{1/2}(2^{i+2}B) \norm{\left|\nabla \Delta^K e^{-(s+\rho(r_B))\Delta} \mathbbm 1_{C_i(B)}b\right|}_{L^2(M)} \frac{ds}{\sqrt s}
\\ &\lesssim 
\int_{\rho (r_B)}^{\infty} \mu^{1/2}(2^{i+2}B) \norm{b}_{L^2(C_i(B))} \frac{ds}{(s+\rho(r_B))^{K+1/2} \sqrt s}
\\ &\lesssim 
2^{-i\varepsilon} \rho^K (r_B) \int_{\rho (r_B)}^{\infty} \frac{ds}{s^{K+1}}
\lesssim 
2^{-i\varepsilon}.
\end{split}
\end{equation*}
For the second term inside the sum, we use Lemma \ref{time derivative}, then
\begin{align*}
&\int_{\rho (r_B)}^{\infty} \norm{\left|\nabla \Delta^K e^{-(s+\rho(r_B))\Delta} \mathbbm 1_{C_i(B)}b\right|}_{L^1(M\backslash 2^{i+2}B)} \frac{ds}{\sqrt s}
\\ &\lesssim 
\sum_{l=i+2}^{\infty} \int_{\rho (r_B)}^{\infty} \mu^{1/p'}(2^{l+1}B) \norm{\left|\nabla \Delta^K e^{-(s+\rho(r_B))\Delta} \mathbbm 1_{C_i(B)}b\right|}_{L^p(C_l(B)} \frac{ds}{\sqrt s}
\\ &\lesssim 
\sum_{l=i+2}^{\infty} \int_{\rho (r_B)}^{\infty} \mu^{1/p'}(2^{l+1}B) \exp\br{-c\br{\frac{\rho(d(C_l(B),C_i(B)))}{s+\rho(r_B)}}^{1/(m-1)}}\frac{ \norm{b}_{L^p(C_i(B)} ds}{\sqrt s \br{s+\rho(r_B)}^{K+1/2}}
\\ &\lesssim 
\sum_{l=i+2}^{\infty} 2^{-i\varepsilon} \rho^K(r_B) \br{\frac{\mu(2^{l+1}B)}{\mu(2^i B)}}^{1/p'}
\int_{\rho (r_B)}^{\infty} \exp\br{-c\br{\frac{\rho(2^l r_B)}{s+\rho(r_B)}}^{1/(m-1)}}\frac{ds}{\sqrt s \br{s+\rho(r_B)}^{K+1/2}}
\\ &\lesssim 
\sum_{l=i+2}^{\infty} 2^{-i\varepsilon} \rho^K(r_B) 2^{(l-i)\nu/p'}
\int_{\rho (r_B)}^{\infty} \br{\frac{s}{\rho(2^l r_B)}}^{c}\frac{ds}{s^{K+1}}
\\ &\lesssim 
\sum_{l=i+2}^{\infty} 2^{-i\varepsilon} \rho^K(r_B) 2^{(l-i)\nu/p'} \frac{1}{\rho^c(2^l r_B)\rho^{K-c}(r_B)}
\\ &\lesssim 
2^{-i\varepsilon}.
\end{align*}
This finishes the proof.
\end{proof}
\bigskip

%%%%%%%%%%%%%%%%%%%%%%%%%%%%%%%%%%%%%%%%%%%%%%%%%%%%%%%%%%%%%%
{\bf Acknowledgements:} 
This work is part of the author's PhD thesis in cotutelle between the Laboratoire de Math\'ematiques, Universit\'e Paris-Sud and the 
Mathematical Science Institute, Australian National University.
The author would like to thank Pascal Auscher for suggesting this topic, and to thank Thierry Coulhon for many discussions and 
suggestions. She is also grateful to Alan McIntosh and Dorothee Frey for helpful discussions. The author was partially supported by the ANR project ``Harmonic analysis at its boundaries'' ANR-12-BS01-0013 and the Australian Research Council (ARC) grant DP130101302.

\bibliographystyle{plain}

\end{document}